\documentclass[a4paper,11pt]{article}
\usepackage{authblk} 
\usepackage{amsmath}
\usepackage{amssymb}
\usepackage[margin=1.2in]{geometry}
\usepackage[round]{natbib}
\usepackage{xcolor}
\RequirePackage[colorlinks,citecolor=blue,urlcolor=blue]{hyperref}
\usepackage{graphicx}
\usepackage{dsfont}
\usepackage{booktabs}
\usepackage{mathabx}

\usepackage{prodint}

\numberwithin{table}{section}
\numberwithin{figure}{section}
\numberwithin{equation}{section}

\definecolor{darkblue}{rgb}{.2, 0.2,.8}
\definecolor{darkgreen}{rgb}{0,0.5,0.3}
\definecolor{darkred}{rgb}{.8, .1,.1}

\newcommand{\bfalp}{\vect{\alpha}}

\newcommand{\bfT}{\mat{T}}
\newcommand{\bfU}{\mat{U}}
\newcommand{\bfQ}{\mat{Q}}

\newcommand*\dd{\mathop{}\!\mathrm{d}}

\renewcommand{\P }{{\mathbb P}}

\usepackage{amsthm}

\newtheorem{lemma}{Lemma}[section]
\newtheorem{theorem}[lemma]{Theorem}

\newtheorem{definition}[lemma]{Definition}

\newtheorem{example}[lemma]{Example}

\newtheorem{remark}{Remark}[section]

\usepackage{bm}
\newcommand{\vect}[1]{\pmb{#1}}
\newcommand{\mat}[1]{\bm{\bm #1}}

\usepackage{marginnote}

\allowdisplaybreaks

\title{\bfseries Assessing continuous common-shock risk through matrix distributions}

\author[1]{Martin Bladt\thanks{Corresponding author. Email: martinbladt@math.ku.dk}}
\author[2]{Oscar Peralta}
\author[3]{Jorge Yslas}

\affil[1]{Department of Mathematical Sciences, University of Copenhagen, Denmark}
\affil[2]{Department of Actuarial Science and Insurance, ITAM, Mexico}
\affil[3]{Department of Mathematical Sciences, University of Liverpool, United Kingdom}

\date{} 

\begin{document}

\maketitle
\begin{abstract}
We introduce a class of continuous-time bivariate phase-type distributions for modeling dependencies from common shocks. The construction uses continuous-time Markov processes that evolve identically until an internal common-shock event, after which they diverge into independent processes. We derive and analyze key risk measures for this new class, including joint cumulative distribution functions, dependence measures, and conditional risk measures. Theoretical results establish analytically tractable properties of the model. For parameter estimation, we employ efficient gradient-based methods. Applications to both simulated and real-world data illustrate the ability to capture common-shock dependencies effectively. Our analysis also demonstrates that common-shock continuous phase-type distributions may capture dependencies that extend beyond those explicitly triggered by common shocks.
\end{abstract}

\section{Introduction}

Risk dependencies play a central role in insurance risk assessment. Multiple lines of business often experience correlated losses due to shared exposure to external factors or shocks. These common factors may include environmental hazards, regulatory changes, or catastrophic events. For example, natural disasters such as hurricanes or earthquakes can simultaneously trigger substantial claims across different insurance portfolios. Explicitly modeling such scenarios requires capturing both the initial shared influence and the subsequent independent evolution of each affected process. Representing these dependencies accurately is essential in practice for aggregate risk management, setting appropriate insurance premiums or managing portfolios.

Continuous-time Markov processes provide a versatile framework for modeling stochastic dynamics. Among these, phase-type ($\mbox{PH}$) distributions are particularly useful. Defined as the distribution of absorption times in finite-state continuous-time Markov chains, their underlying Markovian structure enables a rich theoretical framework with closed-form expressions for probability density functions, cumulative distribution functions, and other key quantities \cite{neuts1981use,bladt2017matrix}. This has made them popular in reliability theory, queueing systems, and risk modeling, and has led to numerous extensions to multivariate settings. One early approach for multivariate risk modeling using continuous-time Markov processes was provided by \cite{assaf1984multivariate}, where vectors of absorption times into distinct subsets of states was examined. In such models, a continuous-time Markov chain is designed with multiple collections of absorbing states, and each component of the multivariate vector denotes the time to absorption in one of these collections. A second method was proposed by \cite{kulkarni1989new}, who introduced vector-valued accumulated rewards. In this setting, the Markov chain accumulates rewards at specified rates while visiting particular states, and the total reward vector upon absorption follows a multivariate $\mbox{PH}$ distribution. {Recent advances in multivariate phase-type distributions have further expanded their modeling capabilities. The multivariate phase-type by mixtures (mPH) construction of \cite{bladt2023tractable} generates dependence through a shared initial state, while subsequent developments have incorporated mixture-of-experts structures for both frequencies and severities \citep{bladt2023robust,bladt2023PHMOEseverity}. These approaches capture flexible marginal and dependence behaviors, but they do not represent scenarios where a single triggering event simultaneously affects multiple components, which is a feature central to common-shock modeling.}

Capturing dependencies is essential for a realistic assessment of joint risk behavior, particularly in insurance applications where multiple risk components interact. The presence of dependencies has inspired a wide range of actuarial research, from models of multivariate claim counts \cite{Shi2014multiBNcounts, Pechon2018MultiClaimCount, Fung2019MOEcorrelatedcounts} and aggregate losses \cite{lee2012MixEr, willmot2015MixEr}, to multivariate risk measures \cite{Cossette2016MultiTVaR, landsman2016multivariate} and ruin probabilities \cite{Badila2015multiruin, Albrecher2022multiruinLaguerre}. While copulas remain a popular tool for modeling dependence \cite[see e.g.][]{Frees1998copula}, their flexibility often comes at the cost of analytical tractability in continuous-time settings. In contrast, structured stochastic models that encode dependence directly in the underlying dynamics can yield both insight and closed-form results. This has motivated the use of common-shock frameworks across various actuarial domains, including credit risk \cite{lindskog2003common}, capital modeling \cite{furman2010multivariate}, collective risk \cite{meyers2007common}, mortality and longevity modeling \cite{alai2013lifetime, alai2016multivariate}, { claims reserving \cite{avanzi2018common, avanzi2021unbalanced, taylor2023auto}, and optimal reinsurance \cite{yuen2015optimal, ceci2022optimal}}. These models reflect the need to incorporate latent shared risk factors that simultaneously impact multiple components. 

Our work contributes to this growing body of literature by introducing a new family of latent Markovian models that explicitly encode common shocks while retaining the analytical strengths of the $\mbox{PH}$ distribution framework. Concretely, this paper proposes a novel class of continuous-time bivariate $\mbox{PH}$ distributions specifically designed to capture explicit dependencies induced by common shocks, complementing recent developments for discrete-time models \cite{Bladt2025}. The proposed model uses continuous-time Markov processes that initially share a common evolution phase, representing the common shock, and subsequently diverge into independent processes. Dependencies may arise both from the occurrence of the common shock itself and from the shared state at the point of divergence. We name this new class of distributions \emph{common-shock continuous $\mbox{PH}$ distributions} ($\mbox{CSPH}$). The $\mbox{CSPH}$ framework directly addresses practical challenges in insurance risk modeling, particularly when multiple lines of business share latent risk factors. 

As a primary goal in risk management is quantifying the dependence structure between multiple perils, we derive and analyze relevant dependence and risk measures within the $\mbox{CSPH}$ framework. We also provide explicit expressions for joint survival probabilities (the likelihood of simultaneous large claims),  dependence measures such as {Pearson's correlation}, and conditional risk measures (expected losses conditional on certain events). The Markovian structure and matrix-analytic methods is particularly well-suited for this task, as the derivation of explicit analytical results is often available for our model.

Estimating the parameters of $\mbox{CSPH}$ distributions is also possible, although it presents significant challenges, particularly since we are dealing with incomplete data and possibly high-dimensional parameter spaces. Traditional estimation methods, such as the expectation maximization (EM) algorithm, have been effective in simpler $\mbox{PH}$ distribution contexts, treating $\mbox{PH}$ data as partially observed jump processes \citep[cf.][]{asmussen1996fitting}. However, the presence of common shocks significantly complicates computations, often leading to slower convergence and numerical instability. To address these limitations, we instead propose using approximate gradient-based methods for parameter estimation. These methods are computationally efficient, relatively straightforward to implement, parallelizable and scalable, making them well-suited for the $\mbox{CSPH}$ model.

The paper is organized as follows. Section \ref{sec:ph} provides essential background on univariate and multivariate $\mbox{PH}$ distributions, motivating the need for specialized models for common-shock scenarios. In Section \ref{sec:common-shocks}, we formally introduce the bivariate common-shock $\mbox{PH}$ ($\mbox{CSPH}$) distribution and discuss its construction. Section \ref{sec:model-properties} is dedicated to deriving its fundamental distributional properties, including marginal and joint distributions. Further theoretical properties, such as denseness, the size-biased Esscher transform, and a master formula for moment calculations, along with several risk measures derived from it, are explored in Section \ref{sec:further-model-properties}. Section \ref{sec:conditional_dependence} delves into conditional dependence measures, analyzing the dependence structure of the residual components after the common shock. Finally, Section \ref{sec:inference_examples} discusses parameter estimation using gradient-based methods and illustrates the model's application through numerical examples with real-world and synthetic data.

\section{$\mbox{PH}$ Distributions}\label{sec:ph}

This section details the construction of the univariate and multivariate $\mbox{PH}$ distributions referenced in the introduction, and can be considered as background material.

\subsection{Univariate $\mbox{PH}$ Distributions}

A (continuous) univariate $\mbox{PH}$ distribution describes the distribution of the time until absorption in a finite-state continuous-time Markov chain with one absorbing state. Concretely, consider a time-homogeneous continuous-time Markov chain with state space \(\{1,2,\dots,p,p+1\}\), where states \(1,2,\dots,p\) are transient and state \(p+1\) is absorbing. The chain begins in one of the transient states according to an initial distribution and eventually transitions to the absorbing state. The time elapsed until absorption is termed a $\mbox{PH}$-distributed random variable.

Formally, let \(\bm{\alpha}\) be a \(1 \times p\)-dimensional probability vector representing the initial distribution over the \(p\) transient states, and let \(\bm{T}\) be a \(p \times p\) subintensity matrix governing transitions among these transient states. By construction, \(\bm{T}\) has nonpositive diagonal entries and nonnegative off-diagonal entries, with each row summing to a nonpositive value. The off-diagonal entries correspond to transition rates between distinct transient states, while the negative diagonal entries represent exit rates from each transient state.

Let \(\bm{1}\) denote a \(p \times 1\) column vector of ones. The absorption rate vector is then given by \(\bm{t} := -\bm{T}\bm{1}\), whose entries specify the rates at which the process leaves the transient states to enter the absorbing state. Starting from the initial distribution \(\bm{\alpha}\), define \(X\) as the time to absorption. The distribution of \(X\) is referred to as a $\mbox{PH}$ distribution with parameters \((\bm{\alpha}, \bm{T})\), denoted
\[
X \sim \mbox{PH}(\bm{\alpha}, \bm{T}).
\]

Standard matrix-analytic arguments \citep[see][]{bladt2017matrix} yield the cumulative distribution function and probability density function for \(X\). These are
\begin{align*}
F_X(t) &= 1 - \bm{\alpha} e^{\bm{T} t} \bm{1}, \quad t \ge 0, \\
f_X(t) &= \bm{\alpha} e^{\bm{T} t}\bm{t}, \quad t \ge 0, \label{eq:ph_pdf}
\end{align*}
where the matrix exponential of a matrix $\mat{M}$ is defined by the power series
\begin{align*}
	e^{\mat{M}} = \sum_{n= 0}^{\infty} \frac{\mat{M}^{n}}{n!} .
\end{align*}

The moment generating function of \(X\) also admits a closed-form expression as
\[
M_X(s) = \mathbb{E}[e^{sX}] = \bm{\alpha}(-\bm{T} - s\bm{I})^{-1}\bm{t}, \quad \text{for } s \text{ in a neighborhood of } 0,
\]
where $\bm{I}$ denotes the identity matrix of appropriate dimension.

A fundamental property of $\mbox{PH}$ distributions is their denseness in the class of all nonnegative distributions with respect to weak convergence \citep[cf.][]{Asmussen2023}. In other words, any distribution supported on \([0,\infty)\) can be arbitrarily well approximated by a $\mbox{PH}$ distribution, provided that one is willing to use a sufficiently large number of phases \(p\). This flexibility, coupled with closed-form expressions and a Markovian interpretation, makes $\mbox{PH}$ distributions an effective modeling tool in a wide range of applications, including the lifetimes of intricate systems in insurance and finance, see for instance \cite{Ahn2012logPH,ALBRECHER202268}.

\subsection{Multivariate $\mbox{PH}$ Distributions}
As discussed in the introduction, extending $\mbox{PH}$ distributions to the multivariate realm is necessary for modeling dependent risks. While several constructions exist, this paper builds upon the structure of mixture-based models. A relevant construction for comparison is the multivariate $\mbox{PH}$ by mixtures ($\mbox{mPH}$) introduced in \cite{bladt2023tractable}.

In a bivariate $\mbox{mPH}$ model, a random vector $(Y_1, Y_2)$ is constructed by selecting an initial state $J_0 \in \{1, \dots, p\}$ from an initial distribution $\bm{\pi}_{\text{mPH}}$. Conditional on $J_0=j$, the components $Y_1$ and $Y_2$ are independent random variables from univariate $\mbox{PH}$ distributions, $Y_1 \sim \mbox{PH}(\bm{e}_j^\top, \bm{S}_1)$ and $Y_2 \sim \mbox{PH}(\bm{e}_j^\top, \bm{S}_2)$, where $\bm{e}_j^\top$ is the $j$-th standard basis vector. Dependence arises exclusively from the shared initial state. The resulting distribution is denoted $\mbox{mPH}(\bm{\pi}_{\text{mPH}}, \{\bm{S}_1, \bm{S}_2\})$. This construction yields explicit formulas for the joint cumulative distribution function and joint density:
\begin{align*}
F_{Y_1, Y_2}(y_1, y_2) &= \sum_{j=1}^{p} (\bm{\pi}_{\text{mPH}})_j \prod_{i=1}^{2} (1 - \bm{e}_j^\top e^{\bm{S}_i y_i} \bm{1}), \quad y_1, y_2 \ge 0, \\
f_{Y_1, Y_2}(y_1, y_2) &= \sum_{j=1}^{p} (\bm{\pi}_{\text{mPH}})_j \prod_{i=1}^{2} (\bm{e}_j^\top e^{\bm{S}_i y_i} \bm{s}_i), \quad y_1, y_2 \ge 0,
\end{align*}
{where} $(\bm{\pi}_{\text{mPH}})_j$ denotes the $j$-th component of $\bm{\pi}_{\text{mPH}}$ and $\bm{s}_i = -\bm{S}_i \bm{1}$. As will be shown, the dependence structure of our $\mbox{CSPH}$ model, conditional on the common shock, is related to this class.

\section{Bivariate Common-Shock $\mbox{PH}$ Distributions}\label{sec:common-shocks}

To model correlated failures in insurance and risk contexts, we incorporate a common-shock segment into the $\mbox{PH}$ framework. Following \cite{Bladt2025}, we adapt their discrete-time model to a continuous-time Markovian setting.

Consider two interconnected processes, \(J_1 = \{J_1(t)\}_{t \ge 0}\) and \(J_2 = \{J_2(t)\}_{t \ge 0}\), both evolving on a shared state space \(\mathcal{E} \cup \mathcal{S}\). States in \(\mathcal{E}\) represent the pre-shock phase, while states in \(\mathcal{S}\) represent the post-shock phase. While both processes remain in \(\mathcal{E}\), including the moment of first entry into \(\mathcal{S}\), they coincide with a joint process \(J_{1,2} = \{J_{1,2}(t)\}_{t \ge 0}\), which evolves according to a subintensity matrix of the block form
\[
\begin{pmatrix} \bm{T} & \bm{U} \\ \bm{0} & \bm{0} \end{pmatrix},
\]
where \(\bm{T}\) governs transitions within \(\mathcal{E}\), and \(\bm{U}\) governs transitions from \(\mathcal{E}\) to \(\mathcal{S}\). Here, and throughout the paper, \(\bm{0}\) denotes a matrix or vector of zeros of appropriate dimension. The condition \(\bm{T}\bm{1} + \bm{U}\bm{1} = \bm{0}\) guarantees that the process will eventually exit \(\mathcal{E}\) to a state in \(\mathcal{S}\).
The corresponding exit time is defined as
\[
\tau_{1,2} = \inf\{ t \ge 0 : J_{1,2}(t) \notin \mathcal{E} \},
\]
and represents the moment of the common shock. Note that this shock is not an exogenous or independent event, but rather emerges from the Markovian evolution.

We set \(J_1(t) = J_2(t) = J_{1,2}(t)\) for all \(t \le \tau_{1,2}\), so both processes evolve identically until the internal shock mechanism is triggered. Upon entering \(\mathcal{S}\), the joint process \(J_{1,2}\) remains absorbed in the state it reaches at time \(\tau_{1,2}\), since we focus on the dynamics up to the shock. At this point, each process \(J_1\) and \(J_2\) continues independently, starting from the common-shock state $K = J_{1,2}(\tau_{1,2})$. Each enters its own copy of the states \(\mathcal{S}\) and evolves under the subintensity matrices \(\bm{Q}_1\) and \(\bm{Q}_2\), respectively, until absorption into a final state outside $\mathcal{E} \cup \mathcal{S}$. The time spent by process $J_i$ in $\mathcal{S}$ after the shock is denoted $\tilde{\tau}_i$. Note that the total time until absorption for process $J_i$ starting from $\mathcal{E}$ would be $\tau_i = \tau_{1,2} + \tilde{\tau}_i$. This total time follows a standard univariate $\mbox{PH}$ distribution with initial distribution $(\bm{\alpha}, \bm{0})$ over $\mathcal{E} \cup \mathcal{S}$ and subintensity matrix
\[
\begin{pmatrix} \bm{T} & \bm{U} \\ \bm{0} & \bm{Q}_i \end{pmatrix}.
\]

The dependence between \(\tau_1\) and \(\tau_2\) arises from the shared time \(\tau_{1,2}\) and the common state $K$ entered at that time. To increase modeling flexibility, the construction allows each component to register the impact of the shock differently through scaling factors $a_1, a_2 > 0$. The variables of interest are defined as
\[
X_i = a_i \tau_{1,2} + \tilde{\tau}_i, \quad i=1,2.
\]
When components have different sensitivities to the shock, this structure accommodates asymmetric responses. When \(a_1 = a_2 = 1\), both components are equally affected, and $X_i = \tau_i$, recovering the total absorption time mentioned above.

\begin{definition}[Bivariate Common-Shock $\mbox{PH}$ Distribution]\label{def:csph}
A bivariate random vector \((X_1, X_2)\) follows a {bivariate common-shock $\mbox{PH}$ ($\mbox{CSPH}$) distribution} with parameters \((\bm{\alpha}, \bm{T}, \bm{U}, \bm{Q}_1, \bm{Q}_2, a_1, a_2)\), where $\bm{\alpha}$ is the initial distribution on $\mathcal{E}$, if two underlying continuous-time Markov processes \(J_1\) and \(J_2\) evolve identically on $\mathcal{E}$ according to $\bm{T}$ starting from $\bm{\alpha}$, until a common exit time $\tau_{1,2}$ where they transition to a state $K \in \mathcal{S}$ according to rates in $\bm{U}$. Subsequently, for $t > \tau_{1,2}$, $J_1$ and $J_2$ evolve independently on $\mathcal{S}$ starting from $K$, governed by subintensity matrices $\bm{Q}_1$ and $\bm{Q}_2$ respectively, until absorption. Let $\tilde{\tau}_i$ be the time spent by $J_i$ in $\mathcal{S}$. The random variables \(X_1\) and \(X_2\) are defined as
\[
X_i = a_i \tau_{1,2} + \tilde{\tau}_i, \quad i=1,2.
\]
\end{definition}

An important property of this construction is that the marginal distributions of $X_1$ and $X_2$ are themselves univariate $\mbox{PH}$. To see this, consider the absorption time of a related Markov process.  The process begins in the pre-shock phase $\mathcal{E}$, with initial distribution $\bm{\alpha}$. The component $a_i \tau_{1,2}$ corresponds to the absorption time from $\mathcal{E}$ where the rates within $\bm{T}$ and exits $\bm{U}$ are scaled by $1/a_i$. The subsequent time $\tilde{\tau}_i$ corresponds to the absorption time from $\mathcal{S}$ governed by $\bm{Q}_i$, starting from the state $K$ determined by the (scaled) exit from $\mathcal{E}$. Combining these, $X_i$ represents the total absorption time of a Markov process on the combined state space $\mathcal{E} \cup \mathcal{S}$ governed by the subintensity matrix
\[
\bm{G}_i = \begin{pmatrix} \bm{T}/a_i & \bm{U}/a_i \\ \bm{0} & \bm{Q}_i \end{pmatrix},
\]
starting from the initial distribution $(\bm{\alpha}, \bm{0})$, which assigns zero initial probability to the states in $\mathcal{S}$. Therefore,
\begin{align}\label{eq:aux-Xi-PH}
X_i \sim \mbox{PH}((\bm{\alpha}, \bm{0}), \bm{G}_i).
\end{align}
Each marginal component thus retains the analytical tractability associated with $\mbox{PH}$ distributions.

By maintaining essential Markovian and $\mbox{PH}$ properties while incorporating correlated risks through a shared event, the model provides closed-form marginal distributions and tractable matrix-analytic methods. The $\mbox{CSPH}$ framework enables the derivation of joint tail probabilities, conditional risk measures, and moment-based metrics. The scaling factors $a_1$ and $a_2$ accommodate different degrees of susceptibility, modeling scenarios where one line of business experiences a shock with a different intensity than another. The $\mbox{CSPH}$ distribution thus provides a framework for analyzing bivariate risks under a single random shock event while maintaining compatibility with standard matrix-analytic techniques for parameter estimation and risk analysis.

\section{Fundamental Properties of Bivariate $\mbox{CSPH}$}
\label{sec:model-properties}

This section presents the fundamental properties of the $\mbox{CSPH}$ family of bivariate distributions. We derive explicit expressions for the marginal and joint probability density functions, the joint cumulative distribution function, and the joint moment generating function. These properties support using the $\mbox{CSPH}$ class as a tractable framework for risk analysis of common-shock scenarios.

\subsection{Shock-Time Distribution and Conditional Residuals}
\label{subsec:shock_residuals}

We decompose the vector $(X_1, X_2)$ by conditioning on the common-shock time $\tau_{1,2}$ and the state $K \in \mathcal{S}$ entered at that time. Recall from Section \ref{sec:common-shocks} that
\[
K = J_{1,2}(\tau_{1,2}) \in \mathcal{S}
\]
is the state reached when the joint process $J_{1,2}$ exits the pre-shock states $\mathcal{E}$. The selection of $K$ is governed by the transition rates in the matrix $\bm{U}$. Upon entering state $K$ at time $\tau_{1,2}$, the individual processes $J_1$ and $J_2$ proceed independently within $\mathcal{S}$, governed by subintensity matrices $\bm{Q}_1$ and $\bm{Q}_2$, respectively. The time process $J_i$ subsequently spends in $\mathcal{S}$ until absorption is the residual lifetime $\tilde{\tau}_i$.

A key property is conditional independence. Given the state $K=k$ entered at time $\tau_{1,2}$, the residual lifetimes $\tilde{\tau}_1$ and $\tilde{\tau}_2$ are independent of each other and of $\tau_{1,2}$. The conditional distribution of $\tilde{\tau}_i$ given $\{K=k\}$ is univariate $\mbox{PH}$ with {representation} $\mbox{PH}(\bm{e}_k^\top, \bm{Q}_i)$.

We characterize the distributions of the components involved in the decomposition of $(X_1, X_2)$ via $(\tau_{1,2}, K, \tilde{\tau}_1, \tilde{\tau}_2)$.

\medskip
\noindent\textbf{Common-shock time \(\tau_{1,2}\) and state \(K\).}
The joint distribution of $(\tau_{1,2}, K)$ is described by the defective probability density function $f_{\tau_{1,2}}^{(k)}(t)$, representing the density associated with the shock occurring at time $t$ and the process entering state $k \in \mathcal{S}$. This is given by the standard $\mbox{PH}$ formula for absorption into state $k$ from $\mathcal{E}$:
\begin{equation} \label{eq:defective_density_shock}
f_{\tau_{1,2}}^{(k)}(t)
=
\bm{\alpha}
e^{\bm{T}t}
\bm{u}^{(k)},
\quad
t \ge 0,
\end{equation}
where 
\(\bm{u}^{(k)} = \bm{U}\bm{e}_k\) is the $k$-th column of \(\bm{U}\).
The marginal probability density function of the common-shock time \(\tau_{1,2}\) is obtained by summing over all possible destination states $k \in \mathcal{S}$:
\begin{equation*} \label{eq:marginal_density_shock}
f_{\tau_{1,2}}(t)
=
\sum_{k \in \mathcal{S}}
f_{\tau_{1,2}}^{(k)}(t)
=
\bm{\alpha}
e^{\bm{T}t}
\bm{U}\bm{1},\quad
t \ge 0.
\end{equation*}
The transform associated with the defective distribution, $\mathbb{E}[e^{s\tau_{1,2}}\mathds{1}_{\{K = k\}}]$, is
\begin{equation} \label{eq:mgf_defective_shock}
M_{\tau_{1,2}}^{(k)}(s)
=
\mathbb{E}[e^{s\tau_{1,2}}\mathds{1}_{\{K = k\}}]
=
\bm{\alpha}
{(- \bm{T}- s\bm{I} )^{-1}}
\bm{u}^{(k)},
\end{equation}
valid for $s$ such that ${(- \bm{T}-s\bm{I})}$ is invertible.

\medskip
\noindent\textbf{Residual lifetimes conditioned on \(\{K = k\}\).}
Given that the common shock leads to state $k \in \mathcal{S}$, the subsequent residual lifetime $\tilde{\tau}_i$ follows the univariate $\mbox{PH}$ distribution $\mbox{PH}(\bm{e}_k^\top, \bm{Q}_i)$. Its probability density function is
\begin{equation} \label{eq:pdf_residual}
f_{\tilde{\tau}_i \mid K=k}(y)
=
\bm{e}_k^\top
e^{\bm{Q}_i y}
\bm{q}_i,
\quad
y \ge 0,
\end{equation}
its cumulative distribution function is
\begin{equation} \label{eq:cdf_residual}
F_{\tilde{\tau}_i \mid K=k}(y)
=
\P(\tilde{\tau}_i \le y \mid K=k) = 1 - \bm{e}_k^\top e^{\bm{Q}_i y} \bm{1}, \quad y \ge 0,
\end{equation}
and its moment generating function is
\begin{equation} \label{eq:mgf_residual}
M_{\tilde{\tau}_i \mid K=k}(s)
=
\mathbb{E}[e^{s \tilde{\tau}_i} \mid K=k]
=
\bm{e}_k^\top
{ (- \bm{Q}_i-s\bm{I} )^{-1}}
\bm{q}_i,
\end{equation}
where \(\bm{q}_i := -\bm{Q}_i\bm{1}\) is the vector of absorption rates from $\mathcal{S}$ for process $J_i$. As noted, \(\tilde{\tau}_1\) and \(\tilde{\tau}_2\) are conditionally independent given $K=k$.

\subsection{Joint Density}
\label{subsec:joint_density}

The joint probability density function $f_{X_1, X_2}(z_1, z_2)$ for the $\mbox{CSPH}$ distributed random vector $(X_1, X_2) = (a_1\tau_{1,2}+\tilde{\tau}_1, a_2\tau_{1,2}+\tilde{\tau}_2)$ can be expressed using a block-matrix exponential representation.

Let $z_1 \curlywedge z_2 = (z_1/a_1)\wedge (z_2/a_2) = \min(z_1/a_1, z_2/a_2)$. We define the matrix function $\bm{V}$ from Van Loan's formula \citep{van1978computing} as
\begin{align}
\bm{V}[\bm{A}, \bm{B}, \bm{C}, u] = \begin{pmatrix} \bm{I} & \bm{0} \end{pmatrix} \exp\left( \begin{pmatrix} \bm{A} & \bm{B} \\ \bm{0} & \bm{C} \end{pmatrix}u \right) \begin{pmatrix} \bm{0} \\ \bm{I} \end{pmatrix},\quad u\ge 0,\label{eq:van-loan-func-V}
\end{align}
which evaluates the integral $\int_0^{u} e^{\bm{A}t} \bm{B} e^{\bm{C}(u-t)} \dd t$. The notation $\otimes$ denotes the Kronecker product and $\oplus$ denotes the Kronecker sum.

\begin{theorem}[Joint Density] \label{thm:joint_density_scaled_block}
Let \(g(t, y_1, y_2)\) denote the trivariate probability density function of $(\tau_{1,2},\,\tilde{\tau}_1,\,\tilde{\tau}_2)$. Then, for $t, y_1, y_2 \ge 0$,
\begin{align}\label{eq:trivariate_density_g}
g(t, y_1, y_2)=\bm{\alpha} e^{\bm{T} t} \bm{P}\Bigl( (e^{\bm{Q}_1 y_1}\bm{q}_1) \otimes (e^{\bm{Q}_2 y_2}\bm{q}_2) \Bigr),
\end{align}
where $\bm{q}_i$ is defined in (\ref{eq:mgf_residual}), $\bm{u}^{(k)} = \bm{U} \bm{e}_k$, and
\[
\bm{P}=\sum_{k\in\mathcal{S}} \bm{u}^{(k)} (\bm{e}_k^\top \otimes \bm{e}_k^\top).
\]
Let $f_{X_1, X_2}(z_1, z_2)$ be the bivariate probability density function of $(X_1, X_2)$. For $z_1, z_2 \ge 0$, this density is given by
\[
f_{X_1, X_2}(z_1, z_2)
=
\bm{\alpha} \bm{V}[\bm{T}, \bm{P}, (a_1\bm{Q}_1)\oplus (a_2\bm{Q}_2), z_1\curlywedge z_2]
\Bigl( \bm{q}_1(z_1, z_2) \otimes \bm{q}_2(z_1, z_2) \Bigr),
\]
where for $i=1,2$,
\begin{align*}
\bm{q}_i(z_1, z_2)&=e^{\bm{Q}_i(z_i-a_i(z_1\curlywedge z_2))} \bm{q}_i.
\end{align*}
\end{theorem}

\begin{proof}
We derive the expression for the trivariate probability density function \(g(t, y_1, y_2)\). Condition on the event that absorption occurs in state \(k\) at time \(\tau_{1,2}=t\). The defective probability density function is $f^{(k)}_{\tau_{1,2}}(t)$ given by (\ref{eq:defective_density_shock}). Given this event, $\tilde{\tau}_1$ and $\tilde{\tau}_2$ are independent with conditional probability density functions $f_{\tilde{\tau}_1 \mid K=k}(y_1)$ and $f_{\tilde{\tau}_2 \mid K=k}(y_2)$ from (\ref{eq:pdf_residual}).
By the law of total probability, summing over all \(k\in\mathcal{S}\) gives
\begin{align*}
g(t, y_1, y_2)
&=\sum_{k\in\mathcal{S}} f^{(k)}_{\tau_{1,2}}(t) f_{\tilde{\tau}_1 \mid K=k}(y_1)f_{\tilde{\tau}_2 \mid K=k}(y_2) \\
&= \bm{\alpha} e^{\bm{T} t} \left(\sum_{k\in\mathcal{S}}\bm{u}^{(k)} (\bm{e}_k^\top \otimes \bm{e}_k^\top)\right)
\big(( e^{\bm{Q}_1 y_1} \bm{q}_1 )\otimes (e^{\bm{Q}_2 y_2} \bm{q}_2 )\big) \\
&= \bm{\alpha} e^{\bm{T} t} \bm{P}\big(( e^{\bm{Q}_1 y_1} \bm{q}_1 )\otimes (e^{\bm{Q}_2 y_2} \bm{q}_2 )\big).
\end{align*}
Since \(X_i=a_i\tau_{1,2}+\tilde{\tau}_i\), the bivariate probability density function $f_{X_1, X_2}(z_1, z_2)$ can be obtained by integrating over all possible realizations of the common shock.  Specifically, the corresponding range of $\tau_{1,2}$ is $0 \le t \le z_1 \wedge z_2$, yielding:
\[
f_{X_1, X_2}(z_1, z_2)=\int_0^{z_1\curlywedge z_2} g(t, z_1-a_1 t, z_2-a_2 t) \dd t.
\]
Substituting \(g\),
\[
f_{X_1, X_2}(z_1, z_2)=\int_0^{z_1\curlywedge z_2} \bm{\alpha} e^{\bm{T} t}\bm{P}\big(( e^{\bm{Q}_1 (z_1-a_1t)} \bm{q}_1 )\otimes (e^{\bm{Q}_2 (z_2-a_2t)} \bm{q}_2 )\big) \dd t.
\]
Giving that $e^{\bm{Q}_i(z_i-a_i t)} = e^{\bm{Q}_i(z_i-a_i(z_1\curlywedge z_2))} e^{a_i\bm{Q}_i((z_1\curlywedge z_2)-t)}$ and using $\bm{q}_i(z_1, z_2)$ as defined above, the term in the integral becomes
\[
\bm{\alpha} e^{\bm{T} t} \bm{P} \Bigl( (e^{a_1\bm{Q}_1((z_1\curlywedge z_2)-t)} \bm{q}_1(z_1, z_2) ) \otimes (e^{a_2\bm{Q}_2((z_1\curlywedge z_2)-t)} \bm{q}_2(z_1, z_2) ) \Bigr).
\]
Applying the identity $e^{a_1\bm{A}r} \otimes e^{a_2\bm{B}r} = e^{((a_1\bm{A})\oplus(a_2\bm{B}))r}$, the expression rewrites as
\[
\bm{\alpha} e^{\bm{T} t} \bm{P} e^{((a_1\bm{Q}_1)\oplus(a_2\bm{Q}_2))( z_1\curlywedge z_2 - t)} (\bm{q}_1(z_1, z_2) \otimes \bm{q}_2(z_1, z_2)).
\]
The integral is therefore
\[
\bm{\alpha} \left( \int_0^{z_1\curlywedge z_2} e^{\bm{T} t} \bm{P} e^{((a_1\bm{Q}_1)\oplus(a_2\bm{Q}_2))((z_1\curlywedge z_2)-t)} \dd t \right) (\bm{q}_1(z_1, z_2) \otimes \bm{q}_2(z_1, z_2)).
\]
Employing (\ref{eq:van-loan-func-V}), this becomes
$\bm{\alpha} \bm{V}[\bm{T}, \bm{P}, (a_1\bm{Q}_1)\oplus (a_2\bm{Q}_2), z_1\curlywedge z_2] (\bm{q}_1(z_1, z_2) \otimes \bm{q}_2(z_1, z_2))$.
\end{proof}

\subsection{Joint Cumulative Distribution Function}
\label{subsec:joint_cdf}

The closed-form expression for the joint cumulative distribution function $F_{X_1, X_2}(z_1, z_2) = \P(X_1 \le z_1, X_2 \le z_2)$ for the scaled variables $X_i = a_i \tau_{1,2} + \tilde{\tau}_i$ also uses the notation $z_1 \curlywedge z_2$ and the function $\bm{V}[\bm{A}, \bm{B}, \bm{C}, u]$ as defined in (\ref{eq:van-loan-func-V}).

\begin{theorem}[Joint Cumulative Distribution Function] \label{thm:joint_cdf_scaled_block_corrected}
Let $(X_1, X_2)$ be a $\mbox{CSPH}$ distributed random vector with parameters $(\bm{\alpha}, \bm{T}, \bm{U}, \bm{Q}_1, \bm{Q}_2, a_1, a_2)$. For $z_1, z_2 \ge 0$, define \[\bm{S}_i(z_1, z_2) = e^{\bm{Q}_i (z_i - a_i (z_1 \curlywedge z_2))} \bm{1}.\]
The joint cumulative distribution function $F_{X_1, X_2}(z_1, z_2)$ is given by
\begin{align} F_{X_1, X_2}(z_1, z_2) = &\bm{\alpha} \bm{V}[\bm{T}, \bm{P}, \bm{0} \oplus \bm{0}, z_1 \curlywedge z_2](\bm{1} \otimes \bm{1}) \nonumber\\
&- \bm{\alpha} \bm{V}[\bm{T}, \bm{P}, a_1\bm{Q}_1 \oplus \bm{0}, z_1 \curlywedge z_2](\bm{S}_1(z_1, z_2) \otimes \bm{1}) \nonumber\\
&- \bm{\alpha} \bm{V}[\bm{T}, \bm{P}, \bm{0} \oplus a_2\bm{Q}_2, z_1 \curlywedge z_2](\bm{1} \otimes \bm{S}_2(z_1, z_2)) \nonumber\\
&+ \bm{\alpha} \bm{V}[\bm{T}, \bm{P}, (a_1\bm{Q}_1)\oplus(a_2\bm{Q}_2), z_1 \curlywedge z_2](\bm{S}_1(z_1, z_2) \otimes \bm{S}_2(z_1, z_2)).\label{eq:aux-joint-2}
\end{align}
\end{theorem}

\begin{proof}
The joint cumulative distribution function is given by $F_{X_1, X_2}(z_1, z_2) = \P(X_1 \leq z_1, X_2 \leq z_2)$. Conditioning on $(\tau_{1,2}, K)$ leads to the integral form
\begin{align}\label{eq:aux-joint-1}
F_{X_1, X_2}(z_1, z_2) = \int_{t=0}^{z_1 \curlywedge z_2} \sum_{k \in \mathcal{S}} f_{\tau_{1,2}}^{(k)}(t) F_{\tilde{\tau}_1 \mid K=k}(z_1 - a_1 t) F_{\tilde{\tau}_2 \mid K=k}(z_2 - a_2 t)  dt,
\end{align}
where $f_{\tau_{1,2}}^{(k)}(t)$ is given by (\ref{eq:defective_density_shock}) and $F_{\tilde{\tau}_i \mid K=k}(y)$ by (\ref{eq:cdf_residual}). Let 
\[S_{\tilde{\tau}_i \mid K=k}(y) = 1 - F_{\tilde{\tau}_i \mid K=k}(y) = \bm{e}_k^\top e^{\bm{Q}_i y} \bm{1}.\] The product of conditional cumulative distribution functions is $(1 - S_{\tilde{\tau}_1 \mid K=k}(z_1 - a_1 t))(1 - S_{\tilde{\tau}_2 \mid K=k}(z_2 - a_2 t))$. Expanding this product and substituting into \eqref{eq:aux-joint-1} yields $F_{X_1, X_2}(z_1, z_2) = I_1 - I_2 - I_3 + I_4$ where
\begin{align*}
I_1 &= \int_0^{z_1 \curlywedge z_2} \sum_{k\in\mathcal{S}} (\bm{\alpha} e^{\bm{T}t} \bm{u}^{(k)}) dt \\
I_2 &= \int_0^{z_1 \curlywedge z_2} \sum_{k\in\mathcal{S}} (\bm{\alpha} e^{\bm{T}t} \bm{u}^{(k)}) (\bm{e}_k^\top e^{\bm{Q}_1 (z_1-a_1 t)}\bm{1}) dt \\
I_3 &= \int_0^{z_1 \curlywedge z_2} \sum_{k\in\mathcal{S}} (\bm{\alpha} e^{\bm{T}t} \bm{u}^{(k)}) (\bm{e}_k^\top e^{\bm{Q}_2 (z_2-a_2 t)}\bm{1}) dt \\
I_4 &= \int_0^{z_1 \curlywedge z_2} \sum_{k\in\mathcal{S}} (\bm{\alpha} e^{\bm{T}t} \bm{u}^{(k)}) (\bm{e}_k^\top e^{\bm{Q}_1 (z_1-a_1 t)}\bm{1}) (\bm{e}_k^\top e^{\bm{Q}_2 (z_2-a_2 t)}\bm{1}) dt.
\end{align*}
Each integral $I_j$ is evaluated using the matrix function $\bm{V}[\bm{A}, \bm{B}, \bm{C}, u]$ defined in (\ref{eq:van-loan-func-V}). This evaluation involves factoring terms 
\[e^{\bm{Q}_i (z_i - a_i t)}\bm{1}=e^{a_i \bm{Q}_i ((z_1 \curlywedge z_2)-t)} e^{\bm{Q}_i (z_i - a_i (z_1 \curlywedge z_2))} \bm{1}=e^{a_i \bm{Q}_i ((z_1 \curlywedge z_2)-t)} \bm{S}_i(z_1, z_2).\] 
Subsequent application of Kronecker product properties, analogous to those employed in the proof of Theorem \ref{thm:joint_density_scaled_block}, leads to the expressions for $I_j$ that align with \eqref{eq:aux-joint-2}.
\end{proof}

\subsection{Joint Moment Generating Function}
\label{subsec:joint_mgf}

The joint moment generating function provides a compact representation of the moments of the distribution.

\begin{theorem}[Joint Moment Generating Function] \label{thm:joint_mgf_scaled}
The joint moment generating function of the $\mbox{CSPH}$ distributed random vector $(X_1, X_2)$ is given by
\begin{align*}
M_{X_1, X_2}(s_1, s_2) &= \mathbb{E}[e^{s_1 X_1 + s_2 X_2}] \\
&= \sum_{k \in \mathcal{S}} M_{\tau_{1,2}}^{(k)}(a_1 s_1 + a_2 s_2) M_{\tilde{\tau}_1 \mid K=k}(s_1) M_{\tilde{\tau}_2 \mid K=k}(s_2),
\end{align*}
where $M_{\tau_{1,2}}^{(k)}(s)$ and $M_{\tilde{\tau}_i \mid K=k}(s)$ are the transforms defined in (\ref{eq:mgf_defective_shock}) and (\ref{eq:mgf_residual}) respectively, valid for $s_1, s_2$ in a neighborhood of 0 such that the underlying matrix inverses exist.
\end{theorem}

\begin{proof}
By definition, $M_{X_1, X_2}(s_1, s_2) = \mathbb{E}[e^{s_1 X_1 + s_2 X_2}]$. Substituting $X_i = a_i \tau_{1,2} + \tilde{\tau}_i$ gives
\[
M_{X_1, X_2}(s_1, s_2) = \mathbb{E}\left[ e^{s_1(a_1 \tau_{1,2} + \tilde{\tau}_1) + s_2(a_2 \tau_{1,2} + \tilde{\tau}_2)} \right]
= \mathbb{E}\left[ e^{(a_1 s_1 + a_2 s_2)\tau_{1,2}} e^{s_1 \tilde{\tau}_1} e^{s_2 \tilde{\tau}_2} \right].
\]
Using the law of total expectation, conditioning on the state $K=k$ entered at time $\tau_{1,2}$, results in
\[
M_{X_1, X_2}(s_1, s_2) = \sum_{k \in \mathcal{S}} \mathbb{E}\left[ e^{(a_1 s_1 + a_2 s_2)\tau_{1,2}} e^{s_1 \tilde{\tau}_1} e^{s_2 \tilde{\tau}_2} \mathds{1}_{\{K=k\}} \right].
\]
By the conditional independence property, given $K=k$, $\tau_{1,2}$ is independent of $(\tilde{\tau}_1, \tilde{\tau}_2)$, and $\tilde{\tau}_1$ is independent of $\tilde{\tau}_2$. Using iterated expectations gives
\begin{align*}
\mathbb{E}&\left[ e^{(a_1 s_1 + a_2 s_2)\tau_{1,2}} e^{s_1 \tilde{\tau}_1} e^{s_2 \tilde{\tau}_2} \mathds{1}_{\{K=k\}} \right]\\ &= \mathbb{E}\left[ \mathds{1}_{\{K=k\}} \mathbb{E}\left[ e^{(a_1 s_1 + a_2 s_2)\tau_{1,2}} e^{s_1 \tilde{\tau}_1} e^{s_2 \tilde{\tau}_2} \mid K=k \right] \right] \\
&= \mathbb{E}\left[ \mathds{1}_{\{K=k\}} \mathbb{E}\left[ e^{(a_1 s_1 + a_2 s_2)\tau_{1,2}} \mid K=k \right] \mathbb{E}\left[ e^{s_1 \tilde{\tau}_1} \mid K=k \right] \mathbb{E}\left[ e^{s_2 \tilde{\tau}_2} \mid K=k \right] \right] \\
&= \mathbb{E}\left[ e^{(a_1 s_1 + a_2 s_2)\tau_{1,2}} \mathds{1}_{\{K=k\}} \right] M_{\tilde{\tau}_1 \mid K=k}(s_1) M_{\tilde{\tau}_2 \mid K=k}(s_2).
\end{align*}
The first term is precisely $M_{\tau_{1,2}}^{(k)}(a_1 s_1 + a_2 s_2)$ from (\ref{eq:mgf_defective_shock}). Summing over $k \in \mathcal{S}$ gives the stated result.
\end{proof}

\section{Further Properties of Bivariate $\mbox{CSPH}$}
\label{sec:further-model-properties}

Building upon the fundamental distributional properties of bivariate $\mbox{CSPH}$ distributions established in Section \ref{sec:model-properties}, this section explores further theoretical aspects and practical applications. We begin by establishing the denseness of the $\mbox{CSPH}$ class, underscoring its capacity to approximate a wide range of bivariate distributions. Subsequently, we explore the incorporation of size-biasing and Esscher transforms, techniques crucial for various applications in risk theory and actuarial science. A key development is a ``master formula'' that provides a unified framework for computing expectations involving these transformed distributions and indicator functions. Finally, we demonstrate the utility of this master formula by deriving expressions for several relevant risk measures, highlighting the analytical tractability of the $\mbox{CSPH}$ framework for risk assessment.

\subsection{Denseness Property}
\label{subsec:density_prop}

We establish that the class of $\mbox{CSPH}$ distributions is dense within the space of all bivariate distributions supported on $[0,\infty)^2$. This property ensures that $\mbox{CSPH}$ models can approximate any such target distribution arbitrarily well, highlighting their flexibility.

\begin{theorem} \label{thm:denseness_csph}
Let $F$ be any bivariate distribution function with support contained in $[0,\infty)^2$. Then, for any fixed scaling factors $a_1 > 0, a_2 > 0$, there exists a sequence of bivariate $\mbox{CSPH}$ distributions with these scaling factors that converges weakly to $F$.
\end{theorem}

\begin{proof}
The proof relies on the denseness property of the $\mbox{mPH}$ class of multivariate PH distributions introduced in \cite{bladt2023tractable}. This class is weakly dense in the set of all multivariate distributions on $[0,\infty)^d$. In our case, we apply this result in the bivariate case ($d=2$).

Let $F_{\text{mPH}}$ be an arbitrary bivariate $\mbox{mPH}$ distribution with parameters $(\bm{\alpha}_{\text{mPH}}, \bm{S}_1, \bm{S}_2)$, where $\bm{\alpha}_{\text{mPH}}$ is the $1 \times p$ initial probability vector and $\bm{S}_i$ is the $p \times p$ subintensity matrix defining the underlying univariate $\mbox{PH}$ structure for the $i$-th coordinate. A random vector $(Y_1, Y_2) \sim F_{\text{mPH}}$ can be realized by first selecting an initial state $J_0 \in \{1, \dots, p\}$ according to $\bm{\alpha}_{\text{mPH}}$, and then letting $Y_1$ and $Y_2$ be independent random variables following $\mbox{PH}(\bm{e}_{J_0}^\top, \bm{S}_1)$ and $\mbox{PH}(\bm{e}_{J_0}^\top, \bm{S}_2)$, respectively.

Now, for any $\lambda > 0$, consider a $\mbox{CSPH}$ distribution constructed as follows. The pre-shock state space is $\mathcal{E} = \{0\}$ (a single state) with initial distribution $\bm{\alpha}_{\text{CSPH}} = (1)$ and subintensity matrix $\bm{T} = (-\lambda)$. The post-shock state space $\mathcal{S} = \{1, \dots, p\}$ consists of the transient states of the target $\mbox{mPH}$ distribution. Transitions from $\mathcal{E}$ to $\mathcal{S}$ are governed by the $1 \times p$ matrix $\bm{U} = \lambda \bm{\alpha}_{\text{mPH}}$. The evolution within $\mathcal{S}$ after the shock is governed by the subintensity matrices $\bm{Q}_1 = \bm{S}_1$ and $\bm{Q}_2 = \bm{S}_2$. The scaling factors $a_1$ and $a_2$ remain fixed as per the theorem statement. This construction satisfies the condition $\bm{T}\bm{1} + \bm{U}\bm{1} = -\lambda + \lambda \bm{\alpha}_{\text{mPH}} \bm{1} = -\lambda + \lambda (1) = 0$, since $\bm{\alpha}_{\text{mPH}}$ is a probability vector. Let $(X_{1,\lambda}, X_{2,\lambda})$ denote a random vector following this specific $\mbox{CSPH}(\bm{\alpha}_{\text{CSPH}}, \bm{T}, \bm{U}, \bm{Q}_1, \bm{Q}_2, a_1, a_2)$ distribution, where $X_{i,\lambda} = a_i \tau_{1,2}^{(\lambda)} + \tilde{\tau}_i^{(\lambda)}$.

By construction, the common-shock time $\tau_{1,2}^{(\lambda)}$ (time spent in state 0) follows an Exponential distribution with rate $\lambda$. The state $K \in \mathcal{S}$ entered at time $\tau_{1,2}^{(\lambda)}$ has its distribution determined by the absorption probabilities from state 0 into $\mathcal{S}$. 
The probability of entering state $k$ is
\[
\P(K=k) = \bm{\alpha}_{\text{CSPH}} (-\bm{T})^{-1} \bm{u}^{(k)} = (1) (\lambda^{-1}) (\lambda (\bm{\alpha}_{\text{mPH}})_k) = (\bm{\alpha}_{\text{mPH}})_k.
\]
Thus, the state $K$ upon entering $\mathcal{S}$ is selected according to the target initial distribution $\bm{\alpha}_{\text{mPH}}$.

Given $K=k$, the residual times $\tilde{\tau}_1^{(\lambda)}$ and $\tilde{\tau}_2^{(\lambda)}$ (representing time spent in $\mathcal{S}$) are independent, following $\mbox{PH}(\bm{e}_k^\top, \bm{S}_1)$ and $\mbox{PH}(\bm{e}_k^\top, \bm{S}_2)$, respectively. 

Now, consider the limit as $\lambda \to \infty$. We have $\tau_{1,2}^{(\lambda)} \sim \text{Exp}(\lambda)$, which converges in probability to 0 (i.e., $\tau_{1,2}^{(\lambda)} \xrightarrow{P} 0$). Let $(Y_1, Y_2)$ be a random vector with the target $\mbox{mPH}$ distribution $F_{\text{mPH}}$. As established, the distribution of $(Y_1, Y_2)$ is obtained by choosing $K$ according to $\bm{\alpha}_{\text{mPH}}$ and then setting $Y_1$ and $Y_2$ as independent draws from $\mbox{PH}(\bm{e}_K^\top, \bm{S}_1)$ and $\mbox{PH}(\bm{e}_K^\top, \bm{S}_2)$, respectively. The random vector $(\tilde{\tau}_1^{(\lambda)}, \tilde{\tau}_2^{(\lambda)})$ (where the superscript $(\lambda)$ indicates association with the $\mbox{CSPH}$ parameter $\lambda$, though the distribution given $K$ is independent of $\lambda$) has the same distribution as $(Y_1, Y_2)$. Since $\tau_{1,2}^{(\lambda)} \xrightarrow{P} 0$ as $\lambda \to \infty$, by Slutsky's theorem,
\[
(X_{1,\lambda}, X_{2,\lambda}) = (a_1 \tau_{1,2}^{(\lambda)} + \tilde{\tau}_1^{(\lambda)}, a_2 \tau_{1,2}^{(\lambda)} + \tilde{\tau}_2^{(\lambda)}) \xrightarrow{d} (a_1 \cdot 0 + Y_1, a_2 \cdot 0 + Y_2) = (Y_1, Y_2),
\]
where $\xrightarrow{d}$ denotes convergence in distribution. The limit $(Y_1, Y_2)$ has the target mPH distribution $F_{\text{mPH}}$.

We have shown that any $\mbox{mPH}$ distribution $F_{\text{mPH}}$ can be obtained as the weak limit of a sequence of $\mbox{CSPH}$ distributions (indexed by $\lambda \to \infty$) with the given fixed scaling factors $a_1, a_2$. Since the class of $\mbox{mPH}$ distributions is weakly dense in the set of all bivariate distributions on $[0,\infty)^2$ \citep[see][]{bladt2023tractable}, it follows that the class of $\mbox{CSPH}$ distributions (with fixed $a_1, a_2 > 0$) is also weakly dense in this space. This completes the proof.
\end{proof}

\begin{remark}\rm
The denseness property implies that $\mbox{CSPH}$ distributions can, in principle, be employed to model any type of dependence structure on $[0,\infty)^2$. This is noteworthy because common-shock modeling often intuitively suggests positive correlation between components. While the interpretation of the common shock naturally lends itself to positive dependence, the $\mbox{CSPH}$ framework is flexible enough to accommodate other dependence structures, including negative dependence, by appropriately choosing the underlying phase-type parameters.
\end{remark}

\subsection{Size-Biased Esscher Transform Distribution and the Master Formula}
\label{sec:size-biased-esscher}

In many applications, especially in risk theory, actuarial science, and financial mathematics, it is useful to consider transformed versions of a probability distribution that emphasize large outcomes or modify the probability measure to reflect risk preferences. Two common such transforms are the size-biased distribution and the Esscher transform. 

A size-biased distribution reweights the original probability density function in proportion to powers of the size of the outcome. If $X$ is a nonnegative random variable with probability density function $f_X(x)$ and finite moment $\mathbb{E}[X^n]$ for some $n \in \{0,1,2,\dots\} =: \mathbb{N}_0$, its size-biased version has probability density function given by
\begin{equation*}
\label{eq:size-biased-density}
f_X^{\text{SB}}(x) = \frac{x^n \, f_X(x)}{\mathbb{E}[X^n]}.
\end{equation*}
This transform naturally arises in contexts where larger outcomes are more likely to be observed (for instance, when sampling by size).

The Esscher transform is a change-of-measure technique that tilts the original probability density function by an exponential factor. For a given random variable $X$ with probability density function $f_X(x)$, the Esscher-transformed probability density function with parameter $\theta$ is defined by
\begin{equation*}
\label{eq:esscher-density}
f_X^{\text{Ess}}(x) = \frac{e^{-\theta x} \, f_X(x)}{\mathbb{E}[e^{-\theta X}]},
\end{equation*}
provided that $\mathbb{E}[e^{-\theta X}]$ is finite. This transform is widely used in premium calculation and option pricing as it adjusts the measure to incorporate a market price of risk. Both transforms can be extended to the multivariate case as follows.

\begin{definition}[Multivariate Size-Biased Esscher Transform]
\label{def:multivariate-sb-esscher}
Let $(\tau_{1,2}, \tilde{\tau}_1, \tilde{\tau}_2)$ be a multivariate random vector with joint probability density function $g(x_{1,2}, x_1, x_2)$ given by \eqref{eq:trivariate_density_g}. Then, for $x_{1,2},x_1,x_2\ge 0$, the multivariate size-biased Esscher-transformed probability density function of $(\tau_{1,2}, \tilde{\tau}_1, \tilde{\tau}_2)$ is defined as
\begin{equation}
\label{eq:sb-esscher-density}
g^{\text{SB-Ess}}(x_{1,2}, x_1, x_2) =
\frac{x_{1,2}^{n_{1,2}} \, x_1^{n_1} \, x_2^{n_2} \, e^{-\theta_{1,2} x_{1,2} -\theta_1 x_1 -\theta_2 x_2} \, g(x_{1,2}, x_1, x_2)}
{\mathbb{E}\big[\tau_{1,2}^{n_{1,2}} \, \tilde{\tau}_1^{n_1} \, \tilde{\tau}_2^{n_2} \, e^{-\theta_{1,2} \tau_{1,2} -\theta_1 \tilde{\tau}_1 -\theta_2 \tilde{\tau}_2}\big]},
\end{equation}
where $n_{1,2}, n_1, n_2 \in \mathbb{N}_0$ are the size-biasing powers, and $\theta_{1,2}, \theta_1, \theta_2$ are the Esscher transform parameters (typically $\theta_j \ge 0$ for risk applications, but defined as long as the expectation in the denominator is finite and non-zero).
\end{definition}

While the numerator in \eqref{eq:sb-esscher-density} involves the probability density function $g$ from \eqref{eq:trivariate_density_g}, its computation for practical purposes, especially the expectation in the denominator, requires further development. The next result exhibits a form for the unnormalized size-biased Esscher transformed probability density function that lends itself well to these computations.

\begin{theorem}\label{thm:sb-Ess1_numerator}
For $n_{1,2}, n_1, n_2 \in \mathbb{N}_0$ and parameters $\theta_{1,2}, \theta_1, \theta_2$, the numerator of the size-biased Esscher transformed probability density function is given by
\begin{align*}
&x_{1,2}^{n_{1,2}} \, x_1^{n_1} \, x_2^{n_2} \, e^{-\theta_{1,2} x_{1,2} -\theta_1 x_1 -\theta_2 x_2} \, g(x_{1,2}, x_1, x_2)\\
&\quad = (n_{1,2}!\,n_1!\,n_2!)\,\bm{\alpha}^{(n_{1,2})} e^{\bm{T}^{(n_{1,2}, \theta_{1,2})} x_{1,2}}
\bm{I}^{(\downarrow,n_{1,2})}
\bm{P} \\
&\qquad \otimes
\left(\bm{I}^{(\leftarrow,n_{1})} \, e^{\bm{Q}_1^{(n_1, \theta_1)} x_1} \bm{q}_1^{(n_1)}\right)
\otimes
\left(\bm{I}^{(\leftarrow,n_{2})}\, e^{\bm{Q}_2^{(n_2, \theta_2)} x_2} \bm{q}_2^{(n_2)}\right),
\end{align*}
where $g(x_{1,2}, x_1, x_2)$ is defined in \eqref{eq:trivariate_density_g}, and we define the following $(N+1)$-block matrices and vectors (where $N$ takes values $n_{1,2}, n_1, n_2$ as appropriate for each component):
\begin{enumerate}
\item The augmented initial vector $\bm{\alpha}^{(N)}$ (for $\tau_{1,2}$) is a $1 \times (N+1)p_0$ row vector (if $\bm{\alpha}$ is $1 \times p_0$) as
\[
\bm{\alpha}^{(N)} =
\begin{pmatrix}
\bm{\alpha} & \bm{0} & \bm{0} & \cdots & \bm{0}
\end{pmatrix}.
\]
\item The augmented subintensity matrix $\bm{M}^{(N,\theta)}$ (representing $\bm{T}^{(n_{1,2},\theta_{1,2})}$, $\bm{Q}_1^{(n_1,\theta_1)}$, or $\bm{Q}_2^{(n_2,\theta_2)}$) is an $(N+1)p \times (N+1)p$ matrix (if $\bm{M}$ is $p \times p$) as
\[
\bm{M}^{(N,\theta)} =
\begin{pmatrix}
\bm{M} - \theta \bm{I} & \bm{I} & \bm{0} & \cdots & \bm{0}\\[1mm]
\bm{0} & \bm{M} - \theta \bm{I} & \bm{I} & \cdots & \bm{0}\\[1mm]
\vdots & \vdots & \ddots & \ddots & \vdots\\[1mm]
\bm{0} & \bm{0} & \cdots & \bm{M} - \theta \bm{I} & \bm{I}\\[1mm]
\bm{0} & \bm{0} & \cdots & \bm{0} & \bm{M} - \theta \bm{I}
\end{pmatrix}.
\]
\item For $i=1,2$, the augmented exit vectors $\bm{q}_i^{(N)}$ (for $\tilde{\tau}_i$) are $(N+1)p_i \times 1$ column vectors (if $\bm{q}_i$ is $p_i \times 1$) as
\[
\bm{q}_i^{(N)} =
\begin{pmatrix}
\bm{0}\\[2mm]\bm{0}\\[2mm]\vdots\\[2mm]\bm{q}_i
\end{pmatrix}.
\]
\item The selector matrices $\bm{I}^{(\downarrow,N)}$ and $\bm{I}^{(\leftarrow,N)}$ are block matrices. $\bm{I}^{(\downarrow,N)}$ is an $(N+1)p \times p$ matrix with $\bm{I}$ (a $p \times p$ identity matrix) in the last block and zeros elsewhere. $\bm{I}^{(\leftarrow,N)}$ is a $p \times (N+1)p$ matrix with $\bm{I}$ in the first block and zeros elsewhere. More specifically, they are given by
\[
\bm{I}^{(\downarrow,N)} =
\begin{pmatrix}
\bm{0} \\ \bm{0} \\ \vdots \\ \bm{I}
\end{pmatrix},
\quad
\bm{I}^{(\leftarrow,N)} =
\begin{pmatrix}
\bm{I} & \bm{0} & \cdots & \bm{0}
\end{pmatrix}.
\]
\end{enumerate}
\end{theorem}
\begin{proof}
The result uses the known property for PH calculus that relates moments and exponential tilting to augmented matrix exponentials. For a subintensity matrix $\bm{M}$, exit vector $\bm{m}$,  and initial vector $\bm{\pi}$, the quantity $x^N e^{-\theta x} \bm{\pi} e^{\bm{M}x} \bm{m}$ can be expressed using augmented matrices. The $(1, N+1)$ block (or $(0,N)$ block depending on indexing) of the matrix exponential $e^{\bm{M}^{(N,\theta)} x}$ is given by $e^{(\bm{M}-\theta\bm{I})x} \frac{x^N}{N!}$ { \citep[cf.][Lemma 3.1]{cheung2022multivariate}}.
Thus, we can write
\[
\bm{\alpha} \left( \bm{I}^{(\leftarrow,N)} e^{\bm{M}^{(N,\theta)} x} \bm{I}^{(\downarrow,N)} \right) = \bm{\alpha} \left( e^{(\bm{M}-\theta\bm{I})x} \frac{x^N}{N!} \right) = \frac{x^N}{N!} e^{-\theta x} \bm{\alpha} e^{\bm{M}x}.
\]
Applying this to $\tau_{1,2}$ with $N=n_{1,2}$, $\bm{M}=\bm{T}$, $\theta=\theta_{1,2}$, the factor $\bm{\alpha}^{(n_{1,2})} e^{\bm{T}^{(n_{1,2}, \theta_{1,2})} x_{1,2}} \bm{I}^{(\downarrow,n_{1,2})}$ effectively computes $ \bm{\alpha} e^{(\bm{T}-\theta_{1,2}\bm{I})x_{1,2}} \frac{x_{1,2}^{n_{1,2}}}{n_{1,2}!}$. Multiplying by $n_{1,2}!$ yields $x_{1,2}^{n_{1,2}} e^{-\theta_{1,2}x_{1,2}} \bm{\alpha} e^{\bm{T}x_{1,2}}$.

For the terms involving $\tilde{\tau}_i$, we have $e^{\bm{Q}_i x_i} \bm{q}_i$. The expression $\bm{I}^{(\leftarrow,n_i)} e^{\bm{Q}_i^{(n_i, \theta_i)} x_i} \bm{q}_i^{(n_i)}$ corresponds to $\left( \bm{I}^{(\leftarrow,n_i)} e^{\bm{Q}_i^{(n_i, \theta_i)} x_i} \bm{I}^{(\downarrow,n_i)} \right) \bm{q}_i$. This evaluates to $\left( e^{(\bm{Q}_i-\theta_i\bm{I})x_i} \frac{x_i^{n_i}}{n_i!} \right) \bm{q}_i$. Multiplying by $n_i!$ gives $x_i^{n_i} e^{-\theta_i x_i} e^{\bm{Q}_i x_i} \bm{q}_i$.

Combining these parts with the structure of $$g(x_{1,2}, x_1, x_2) = \bm{\alpha} e^{\bm{T} x_{1,2}} \bm{P} \left( (e^{\bm{Q}_1 x_1}\bm{q}_1) \otimes (e^{\bm{Q}_2 x_2}\bm{q}_2) \right)$$ and the factors $n_{1,2}!, n_1!, n_2!$ leads to the stated result. The use of Van Loan's integral formula \citep{van1978computing} provides the foundation for the structure of such matrix exponential blocks, although here we appeal to the direct result for this specific block structure.
\end{proof}

These relations yield the following {master formula} for computing joint moments of tilted and truncated versions of $(\tau_{1,2}, \tilde{\tau}_1, \tilde{\tau}_2)$.

\begin{theorem}[Master Formula]\label{thm:master-formula}
Let $\bm{n}=(n_{1,2},n_1,n_2)$ with $n_j \in \mathbb{N}_0$, $\bm{\theta}=(\theta_{1,2},\theta_1,\theta_2)$, and $\bm{y}=(y_{1,2},y_1,y_2)$ with $y_j \ge 0$. Then, assuming convergence of the integrals and invertibility of the relevant matrices,
\begin{align}
\label{eq:master-formula}
&\mathcal{M}_{\bm{n}, \bm{\theta}, \bm{y}} := \mathbb{E}\Biggl[\left(\tau_{1,2}^{n_{1,2}} e^{-\theta_{1,2} \tau_{1,2}}\right)
\left(\tilde{\tau}_1^{\,n_1} e^{-\theta_1 \tilde{\tau}_1}\right)
\left(\tilde{\tau}_2^{\,n_2} e^{-\theta_2 \tilde{\tau}_2}\right)
\, \mathds{1}_{\{\tau_{1,2} > y_{1,2},\, \tilde{\tau}_1 > y_1,\, \tilde{\tau}_2 > y_2\}}\Biggr] \nonumber \\
&\quad = (n_{1,2}!\,n_1!\,n_2!)\,\bm{\alpha}^{(n_{1,2})}
 e^{\bm{T}^{(n_{1,2}, \theta_{1,2})} y_{1,2}}\, \left(-\bm{T}^{(n_{1,2}, \theta_{1,2})}\right)^{-1}
\bm{I}^{(\downarrow,n_{1,2})}
\bm{P} \nonumber \\
&\qquad \otimes \left(\bm{I}^{(\leftarrow,n_{1})} \, e^{\bm{Q}_1^{(n_1, \theta_1)} y_1}\,\left(-\bm{Q}_1^{(n_1, \theta_1)}\right)^{-1}\,\bm{q}_1^{(n_1)} \right) \nonumber \\
&\qquad \otimes \left(\bm{I}^{(\leftarrow,n_{2})} \, e^{\bm{Q}_2^{(n_2, \theta_2)} y_2}\,\left(-\bm{Q}_2^{(n_2, \theta_2)}\right)^{-1}\,\bm{q}_2^{(n_2)} \right).
\end{align}
This is derived by integrating the expression from Theorem \ref{thm:sb-Ess1_numerator} over the region $x_{1,2} > y_{1,2}, x_1 > y_1, x_2 > y_2$, using the property $\int_y^\infty e^{\bm{A}x} \dd x = e^{\bm{A}y}(-\bm{A})^{-1}$ for a nonsingular matrix $\bm{A}$.
\end{theorem}

\begin{remark}\rm
Choosing $\bm{y}=\bm{0}=(0,0,0)$ in \eqref{eq:master-formula} yields the expectation in the denominator of \eqref{eq:sb-esscher-density}. Thus, together with Theorem \ref{thm:sb-Ess1_numerator}, we can compute the size-biased Esscher transformed probability density function $g^{\text{SB-Ess}}$ explicitly. By adjusting the parameters $\bm{n}$ and $\bm{\theta}$, various transforms (size-biasing, Esscher transform, or combinations) are naturally incorporated into this matrix-analytic framework.
\end{remark}

\subsection{Risk Measures Derived from the Master Formula}
\label{sec:risk-measures-master}

The master formula \eqref{eq:master-formula} is a versatile tool that enables the calculation of a range of risk measures associated with the $\mbox{CSPH}$ distribution, particularly for the components $\tau_{1,2}$, $\tilde{\tau}_1$, and $\tilde{\tau}_2$, and sums thereof. The development and analysis of risk measures are crucial in actuarial science (see, e.g., \cite{ArtznerMF1999, goovaerts2010RM}, \cite{Cousin2014MultiCTE, Cossette2016MultiTVaR}). In what follows, $\bm{0}$ as a subscript or argument denotes appropriate zero vectors or scalars, e.g., $\bm{y}=\bm{0}$ means $(y_{1,2},y_1,y_2)=(0,0,0)$. In this subsection, we work with the general form $X_i = a_i \tau_{1,2} + \tilde{\tau}_i$ as defined in Section \ref{sec:common-shocks}, incorporating the scaling factors $a_1, a_2 > 0$.  

\begin{example}\label{ex:risk_measure}
We consider a CSPH model with $|\mathcal{E}| = 3$ and $|\mathcal{S}| = 2$, and parameters given by
\begin{gather*} 
	\bfalp=\left(
	1, \,0,\, 0 \right)\,, \quad  
	\bfT=\left( \begin{array}{cccc}
	-1/2 & 1/4 & 1/8 \\
	1/8  & -5/8 & 1/4   \\
	1/8 & 1/8 & -3/4
	\end{array} \right) \,, 	\quad 
	\bfU=\left( \begin{array}{cc}
	1/10 & 1/40  \\
	1/8 & 1/8   \\
	1/8 & 3/8 
	\end{array} \right) \,, \\
	\bfQ_1=\left( \begin{array}{cccc}
	-3/8 & 3/8  \\
	0  & -3/8 
	\end{array} \right) \,, \quad 
	\bfQ_2=\left( \begin{array}{cccc}
	-1/2 & 1/4  \\
	1/4  & -1/2 
	\end{array} \right) \,, \\ 
	a_1= 2 \,, \quad a_2 = 1\,.
\end{gather*}
This model will be used throughout the remainder of this section to demonstrate the computation and interpretation of the various risk measures associated to CSPH distributions.
\end{example}

\subsubsection{Pearson Correlation Between $X_1$ and $X_2$}

The Pearson correlation coefficient measures the linear dependence between $X_1$ and $X_2$. Its calculation relies on first and second moments, as well as cross-moments. It is defined as
\[
\rho(X_1, X_2) = \frac{\mbox{Cov}(X_1, X_2)}{\sqrt{\mbox{Var}(X_1) \cdot \mbox{Var}(X_2)}} = \frac{\mathbb{E}[X_1 X_2] - \mathbb{E}[X_1] \mathbb{E}[X_2]}{\sqrt{(\mathbb{E}[X_1^2] - (\mathbb{E}[X_1])^2) \cdot (\mathbb{E}[X_2^2] - (\mathbb{E}[X_2])^2) }}.
\]
To compute this, we need various moments, which can be derived from the master formula by setting $\bm{\theta}=\bm{0}$ and $\bm{y}=\bm{0}$
\begin{align*}
\mathbb{E}[X_1] &= \mathbb{E}[a_1 \tau_{1,2} + \tilde{\tau}_1] = a_1 \mathbb{E}[\tau_{1,2}] + \mathbb{E}[\tilde{\tau}_1] \\
&= a_1 \mathcal{M}_{(1,0,0), \bm{0}, \bm{0}} + \mathcal{M}_{(0,1,0), \bm{0}, \bm{0}} \\
\mathbb{E}[X_2] &= \mathbb{E}[a_2 \tau_{1,2} + \tilde{\tau}_2] = a_2 \mathbb{E}[\tau_{1,2}] + \mathbb{E}[\tilde{\tau}_2] \\
&= a_2 \mathcal{M}_{(1,0,0), \bm{0}, \bm{0}} + \mathcal{M}_{(0,0,1), \bm{0}, \bm{0}} \\
\mathbb{E}[X_1^2] &= \mathbb{E}[(a_1 \tau_{1,2} + \tilde{\tau}_1)^2] = a_1^2 \mathbb{E}[\tau_{1,2}^2] + 2a_1 \mathbb{E}[\tau_{1,2}\tilde{\tau}_1] + \mathbb{E}[\tilde{\tau}_1^2] \\
&= a_1^2 \mathcal{M}_{(2,0,0), \bm{0}, \bm{0}} + 2a_1 \mathcal{M}_{(1,1,0), \bm{0}, \bm{0}} + \mathcal{M}_{(0,2,0), \bm{0}, \bm{0}} \\
\mathbb{E}[X_2^2] &= \mathbb{E}[(a_2 \tau_{1,2} + \tilde{\tau}_2)^2] = a_2^2 \mathbb{E}[\tau_{1,2}^2] + 2a_2 \mathbb{E}[\tau_{1,2}\tilde{\tau}_2] + \mathbb{E}[\tilde{\tau}_2^2] \\
&= a_2^2 \mathcal{M}_{(2,0,0), \bm{0}, \bm{0}} + 2a_2 \mathcal{M}_{(1,0,1), \bm{0}, \bm{0}} + \mathcal{M}_{(0,0,2), \bm{0}, \bm{0}} \\
\mathbb{E}[X_1 X_2] &= \mathbb{E}[(a_1 \tau_{1,2} + \tilde{\tau}_1)(a_2 \tau_{1,2} + \tilde{\tau}_2)] \\
&= a_1 a_2 \mathbb{E}[\tau_{1,2}^2] + a_1 \mathbb{E}[\tau_{1,2}\tilde{\tau}_2] + a_2 \mathbb{E}[\tilde{\tau}_1\tau_{1,2}] + \mathbb{E}[\tilde{\tau}_1\tilde{\tau}_2] \\
&= a_1 a_2 \mathcal{M}_{(2,0,0),\bm{0},\bm{0}} + a_1 \mathcal{M}_{(1,0,1),\bm{0},\bm{0}} + a_2 \mathcal{M}_{(1,1,0),\bm{0},\bm{0}} + \mathcal{M}_{(0,1,1),\bm{0},\bm{0}}.
\end{align*}

For the model specified in Example~\ref{ex:risk_measure}, we obtain $\mathbb{E}[X_1] = 12.87$, $\mathbb{E}[X_2] = 8.44$, $\mbox{Var}(X_1) = 69.51$, $\mbox{Var}(X_2) = 30.85$,  and $\rho(X_1, X_2) = 0.6291$. Notably, we have that $\mathbb{E}[\tau_{1,2}] = 4.44$, highlighting the significant contribution of the common shock component in the overall risk of the model. 

\subsubsection{Entropic Risk Measure}
The entropic risk measure, related to exponential utility functions and the Esscher transform, provides a method for evaluating random outcomes by considering the decision-maker's risk aversion across the entire distribution. This measure is significant in the framework of convex risk measures and finds applications in actuarial science and finance \citep{follmer2002convex, ben2007old, follmer2011entropic}. For a parameter $\vartheta > 0$, which reflects the degree of risk aversion, the entropic risk measure for $X_i$ is defined as
\[
\mbox{ERM}_{\vartheta}(X_i) = \frac{1}{\vartheta} \log \mathbb{E}[e^{-\vartheta X_i}], \quad i=1,2.
\]
The expectation $\mathbb{E}[e^{-\vartheta X_i}]$ can be computed using the master formula. For $X_1$, this is
\[
\mathbb{E}[e^{-\vartheta X_1}] = \mathbb{E}[e^{-\vartheta (a_1 \tau_{1,2}+\tilde{\tau}_1)}] = \mathbb{E}[e^{-a_1\vartheta \tau_{1,2}}e^{-\vartheta \tilde{\tau}_1}] = \mathcal{M}_{\bm{0},(a_1\vartheta,\vartheta,0),\bm{0}}.
\]
Similarly for $X_2$, it is
\[
\mathbb{E}[e^{-\vartheta X_2}] = \mathbb{E}[e^{-\vartheta (a_2 \tau_{1,2}+\tilde{\tau}_2)}] = \mathbb{E}[e^{-a_2\vartheta \tau_{1,2}}e^{-\vartheta \tilde{\tau}_2}] = \mathcal{M}_{\bm{0},(a_2\vartheta,0,\vartheta),\bm{0}}.
\]
The conditional version, $\mbox{ERM}_{\vartheta}(X_i \mid \tau_{1,2} > a)$, provides insight into how the risk assessment of $X_i$ changes given that the common shock has not occurred by time $a$. It quantifies the entropic risk measure in a scenario where the initial period of common exposure has extended beyond a certain duration, thereby isolating the risk profile under that specific condition. The conditional entropic risk measure is given by
\[
\mbox{ERM}_{\vartheta}(X_i \mid \tau_{1,2} > a) = \frac{1}{\vartheta} \log \mathbb{E}[e^{-\vartheta X_i} \mid \tau_{1,2} > a] = \frac{1}{\vartheta} \log \left( \frac{\mathbb{E}[e^{-\vartheta X_i}\mathds{1}_{\{\tau_{1,2}>a\}}]}{\P(\tau_{1,2}>a)} \right).
\]
The denominator is $\P(\tau_{1,2}>a) = \mathcal{M}_{\bm{0},\bm{0},(a,0,0)}$. The numerator for $X_1$ is
\[
\mathbb{E}[e^{-\vartheta X_1}\mathds{1}_{\{\tau_{1,2}>a\}}] = \mathbb{E}[e^{-a_1\vartheta \tau_{1,2}}e^{-\vartheta \tilde{\tau}_1}\mathds{1}_{\{\tau_{1,2}>a\}}] = \mathcal{M}_{\bm{0},(a_1\vartheta,\vartheta,0),(a,0,0)}.
\]
And for $X_2$, it is
\[
\mathbb{E}[e^{-\vartheta X_2}\mathds{1}_{\{\tau_{1,2}>a\}}] = \mathbb{E}[e^{-a_2\vartheta \tau_{1,2}}e^{-\vartheta \tilde{\tau}_2}\mathds{1}_{\{\tau_{1,2}>a\}}] = \mathcal{M}_{\bm{0},(a_2\vartheta,0,\vartheta),(a,0,0)}.
\]

Figure~\ref{fig:erm} displays the entropic risk measure $\mbox{ERM}_{\vartheta}(X_1 \mid \tau_{1,2} > a) $ (left) and $\mbox{ERM}_{\vartheta}(X_2 \mid \tau_{1,2} > a) $ (right) for the model specified in Example~\ref{ex:risk_measure}. The plots show the risk measure's behavior across a range of values for the risk aversion parameter $\vartheta$ and for different thresholds $a$ on the common shock. We observe that the entropic risk increases with both $\vartheta$ and $a$, capturing the compounding effect of increased risk sensitivity and more extreme conditioning. Note that the case $a = 0$ corresponds to the unconditional entropic risk.
\begin{figure}[!htbp]
\centering
\includegraphics[width=0.49\textwidth]{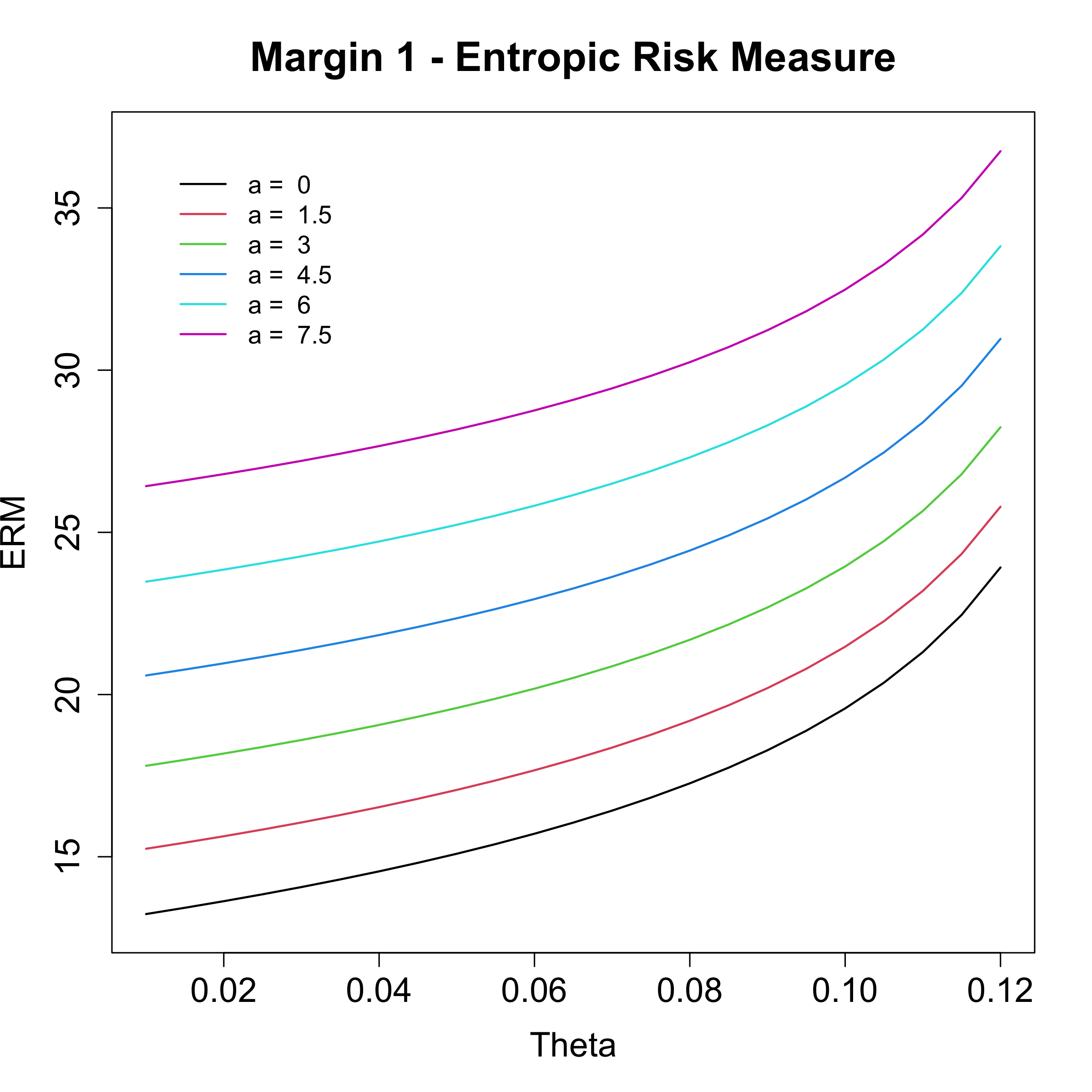}
\includegraphics[width=0.49\textwidth]{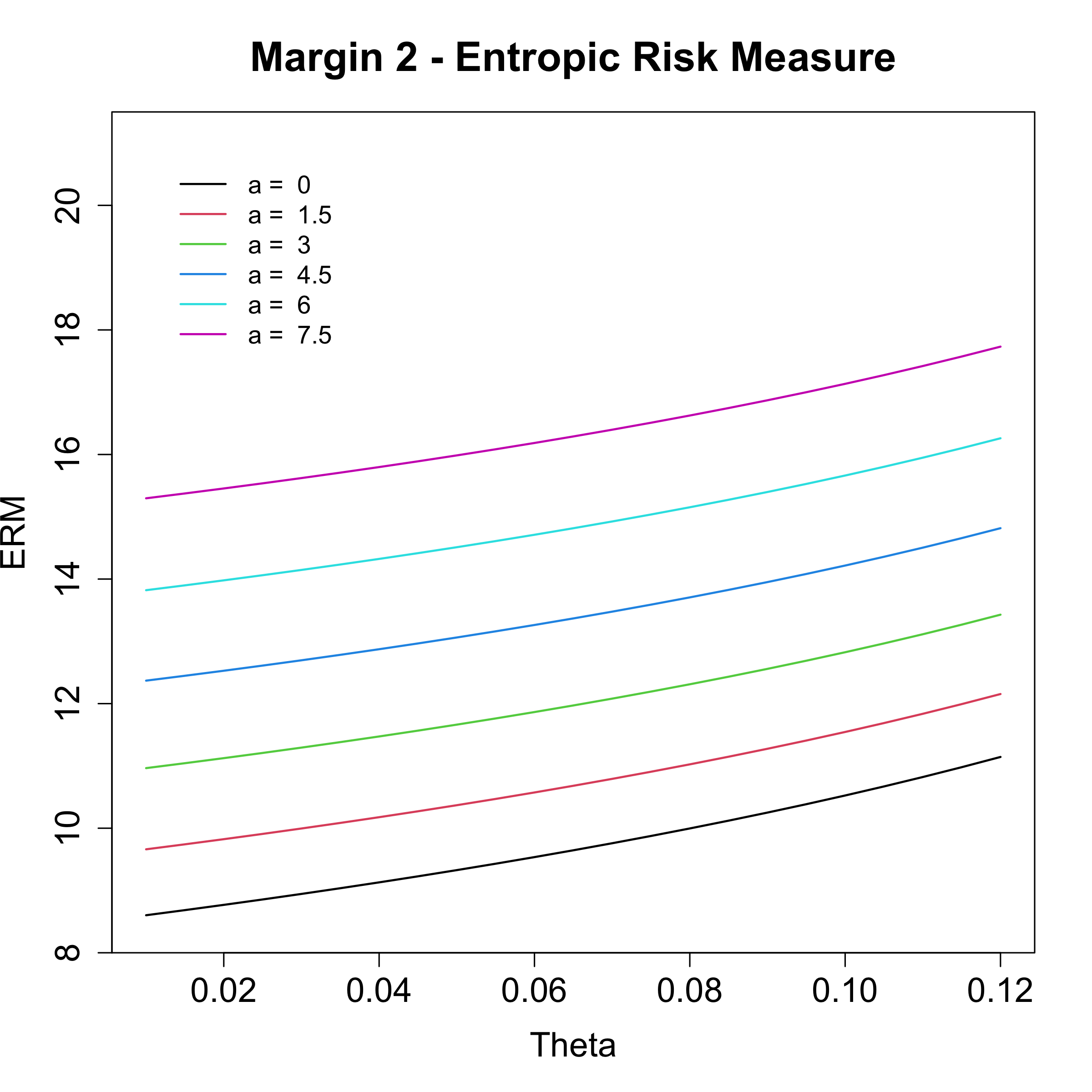}
\caption{Entropic risk measure for the marginal distributions of the CSPH model in Example~\ref{ex:risk_measure}, computed for varying values of $\vartheta$ and conditional on the common shock exceeding a threshold $a = 0, 1.5, 3, 4.5, 6, 7.5$.}
\label{fig:erm}
\end{figure}

\subsubsection{Common-Shock Conditional Value-at-Risk ($\mbox{CV@R}^{\text{CS}}$)}
Recall that the Value-at-Risk ($\mbox{V@R}$) at a confidence level $\alpha \in (0,1)$ for a loss $X_i$ is defined as its $\alpha$-quantile as
\[
\mbox{V@R}_{\alpha}(X_i) = \inf\{x: F_{X_i}(x) \ge \alpha\},
\]
where $F_{X_i}$ is the cumulative distribution function of $X_i$. Computing the $\mbox{V@R}$ typically involves inverting $F_{X_i}$. Since $X_i$ involves $\mbox{PH}$ distributed components (as per \eqref{eq:aux-Xi-PH}), its distribution function is known but generally does not admit a simple analytical inverse. Thus, $\mbox{V@R}_{\alpha}(X_i)$ is usually determined numerically. For example, Table~\ref{tab:var} reports the values of \(\mbox{V@R}_{\alpha}(X_1)\) and \(\mbox{V@R}_{\alpha}(X_2)\) for the CSPH model in Example~\ref{ex:risk_measure} at several confidence levels $\alpha$.

\begin{table}[htbp]
\centering
\caption{Value-at-Risk at different levels for the CSPH model in Example~\ref{ex:risk_measure}.}
\label{tab:var}
\begin{tabular}{lccc}
\toprule
& \multicolumn{3}{c}{Level $\alpha$}   \\ \cline{2-4}
  & 0.95     & 0.975     & 0.99    \\
\midrule
$\mbox{V@R}_{\alpha}(X_1)$ &  28.89 & 33.94 & 40.64 \\
$\mbox{V@R}_{\alpha}(X_2)$ &  19.14 & 22.31 & 26.40\\
\bottomrule
\end{tabular}
\end{table}

In the context of our common-shock framework, we introduce a specific conditional risk measure, termed the Common-Shock Conditional Value-at-Risk, denoted $\mbox{CV@R}^{\text{CS}}_{a}(X_i)$. This measure is defined as the expected value of an individual loss $X_i$ conditioned on the common-shock time $\tau_{1,2}$ exceeding a certain threshold $a \ge 0$ as
\[
\mbox{CV@R}^{\text{CS}}_{a}(X_i) = \mathbb{E}[X_i \mid \tau_{1,2} > a].
\]
This modification is relevant because it directly quantifies how the expectation of an individual loss $X_i$ is revised based purely on information about the persistence of the common-shock phase (i.e., $\tau_{1,2} > a$). Unlike standard Conditional Value-at-Risk definitions that further condition on $X_i$ exceeding its own tail threshold (often $\mbox{V@R}_{\alpha}(X_i)$), $\mbox{CV@R}^{\text{CS}}_{a}(X_i)$ isolates the impact of the common shock's progression. Such a measure is useful for dynamic risk assessment as more information about $\tau_{1,2}$ becomes available, for understanding conditional exposures in systems affected by a common triggering event, or for pricing contracts where payouts depend on the common shock occurring after time $a$. It focuses on the expected severity of $X_i$ under a specific scenario related to the common-shock driver.

The computation of $\mbox{CV@R}^{\text{CS}}_{a}(X_i)$ can be performed using the master formula as
\[
\mbox{CV@R}^{\text{CS}}_{a}(X_i) = \frac{\mathbb{E}[X_i \mathds{1}_{\{\tau_{1,2} > a\}}]}{\P(\tau_{1,2} > a)}.
\]
The denominator is $\P(\tau_{1,2} > a) = \mathcal{M}_{\bm{0},\bm{0},(a,0,0)}$. For $X_1$, the numerator is
\begin{align*}
\mathbb{E}[X_1 \mathds{1}_{\{\tau_{1,2} > a\}}] &= \mathbb{E}[(a_1 \tau_{1,2} + \tilde{\tau}_1) \mathds{1}_{\{\tau_{1,2} > a\}}] \\
&= a_1 \mathbb{E}[\tau_{1,2} \mathds{1}_{\{\tau_{1,2} > a\}}] + \mathbb{E}[\tilde{\tau}_1 \mathds{1}_{\{\tau_{1,2} > a\}}] \\
&= a_1 \mathcal{M}_{(1,0,0),\bm{0},(a,0,0)} + \mathcal{M}_{(0,1,0),\bm{0},(a,0,0)}.
\end{align*}
Similarly for $X_2$, it is
\[
\mathbb{E}[X_2 \mathds{1}_{\{\tau_{1,2} > a\}}] = \mathbb{E}[(a_2 \tau_{1,2} + \tilde{\tau}_2) \mathds{1}_{\{\tau_{1,2} > a\}}] = a_2 \mathcal{M}_{(1,0,0),\bm{0},(a,0,0)} + \mathcal{M}_{(0,0,1),\bm{0},(a,0,0)}.
\]
Thus,
\[
\mbox{CV@R}^{\text{CS}}_{a}(X_1) = \frac{a_1 \mathcal{M}_{(1,0,0),\bm{0},(a,0,0)} + \mathcal{M}_{(0,1,0),\bm{0},(a,0,0)}}{\mathcal{M}_{\bm{0},\bm{0},(a,0,0)}},
\]
and
\[
\mbox{CV@R}^{\text{CS}}_{a}(X_2) = \frac{a_2 \mathcal{M}_{(1,0,0),\bm{0},(a,0,0)} + \mathcal{M}_{(0,0,1),\bm{0},(a,0,0)}}{\mathcal{M}_{\bm{0},\bm{0},(a,0,0)}}.
\]


The behavior of $\mbox{CV@R}^{\text{CS}}_{a}(X_1)$ and $\mbox{CV@R}^{\text{CS}}_{a}(X_2)$ for the model in Example~\ref{ex:risk_measure} is illustrated in the left and right panels of Figure~\ref{fig:cvar}, respectively. As the threshold $a$ increases, we observe a corresponding increase in both risk measures, reflecting the impact on the expected losses when realizations of the common shock are observed.
\begin{figure}[!htbp]
\centering
\includegraphics[width=0.49\textwidth]{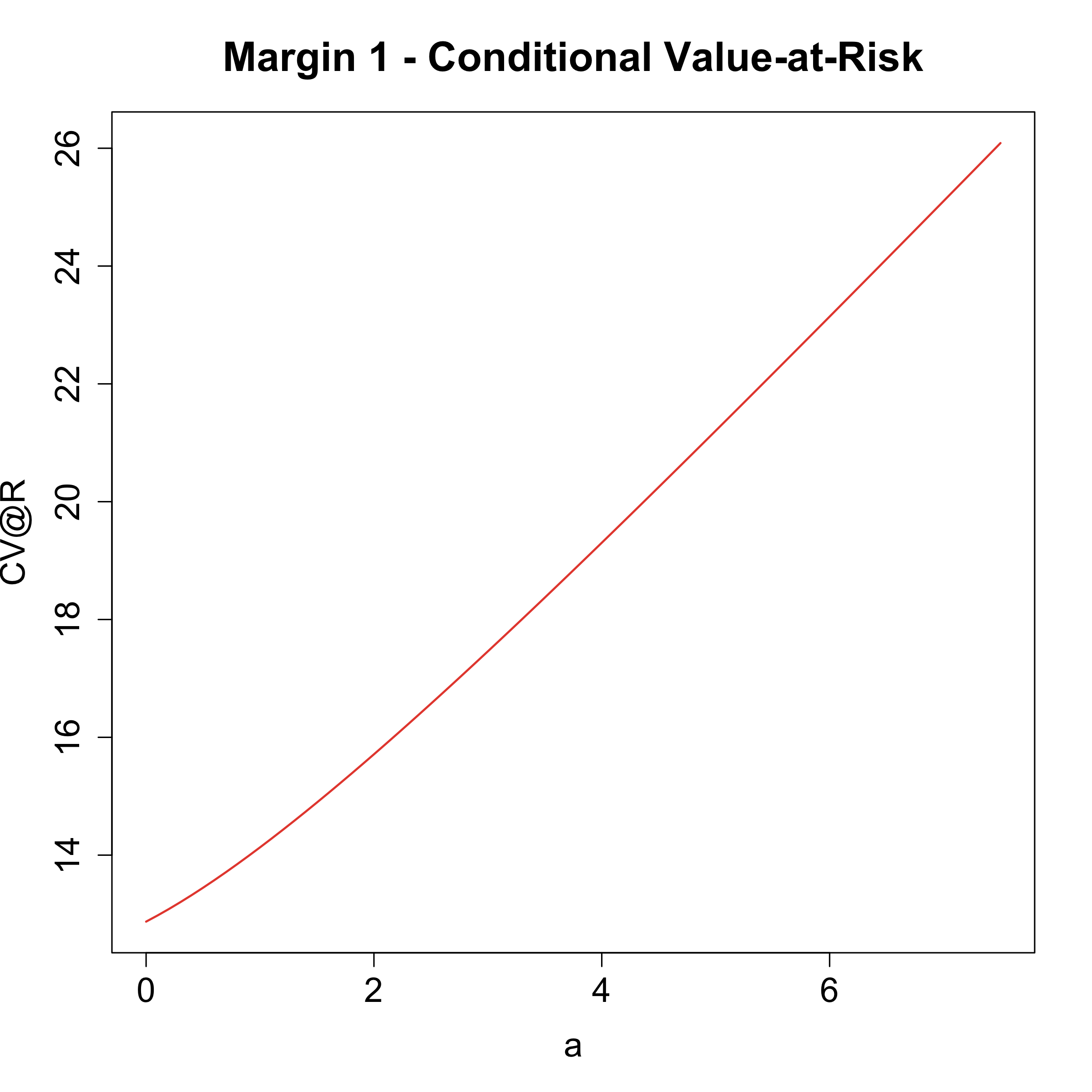}
\includegraphics[width=0.49\textwidth]{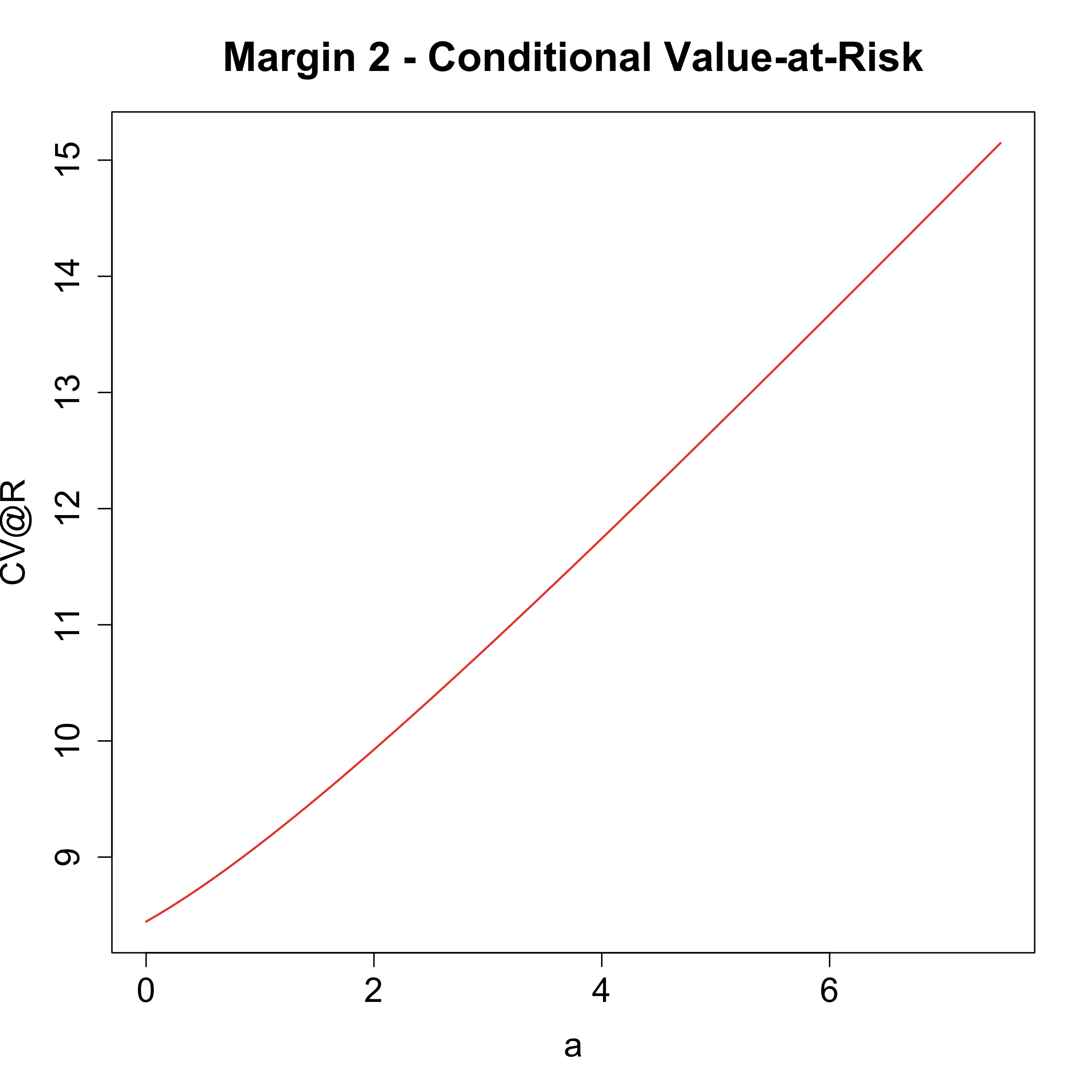}
\caption{CV@R for the marginal distributions of the CSPH model in Example~\ref{ex:risk_measure} plotted as a function of the threshold parameter $a$ affecting the common-shock component $\tau_{1,2}$.}
\label{fig:cvar}
\end{figure}

\subsubsection{Common-Shock Multivariate Tail Conditional Expectation ($\mbox{MTCE}^{\text{CS}}$)}
The multivariate tail conditional expectation (MTCE) is a risk measure that extends the concept of tail conditional expectation to a multivariate setting, often used to understand the behavior of multiple risks when they are jointly in an extreme state (see, e.g., \cite{Cousin2014MultiCTE, landsman2016multivariate, zhu2012asymptotic}). Building on the common-shock conditioning approach used for $\mbox{CV@R}^{\text{CS}}$, we define the Common-Shock Multivariate Tail Conditional Expectation, denoted $\mbox{MTCE}^{\text{CS}}_{a}(X_1, X_2)$. This measures the expected product of $X_1$ and $X_2$ given that the common-shock time $\tau_{1,2}$ exceeds a threshold $a \ge 0$ as
\[
\mbox{MTCE}^{\text{CS}}_{a}(X_1, X_2) = \mathbb{E}[X_1 X_2 \mid \tau_{1,2} > a].
\]
This variant allows for the assessment of the conditional joint impact of $X_1$ and $X_2$ based solely on the persistence of the pre-shock phase. The calculation proceeds as
\[
\mbox{MTCE}^{\text{CS}}_{a}(X_1, X_2) = \frac{\mathbb{E}[X_1 X_2 \cdot \mathds{1}_{\{\tau_{1,2}>a\}}]}{\P(\tau_{1,2} > a)}.
\]
The denominator is $\P(\tau_{1,2} > a) = \mathcal{M}_{\bm{0},\bm{0},(a,0,0)}$. For the numerator, expanding the product gives
\begin{align*}
\mathbb{E}[X_1 X_2 \cdot \mathds{1}_{\{\tau_{1,2}>a\}}] &= \mathbb{E}[(a_1 \tau_{1,2} + \tilde{\tau}_1)(a_2 \tau_{1,2} + \tilde{\tau}_2)\mathds{1}_{\{\tau_{1,2}>a\}}] \\
&= \mathbb{E}[(a_1 a_2 \tau_{1,2}^2 + a_1 \tau_{1,2}\tilde{\tau}_2 + a_2 \tilde{\tau}_1\tau_{1,2} + \tilde{\tau}_1\tilde{\tau}_2)\mathds{1}_{\{\tau_{1,2}>a\}}].
\end{align*}
This can be expressed with the master formula by setting $y_1=0, y_2=0$ in the conditioning vector $\bm{y}=(a,0,0)$, which yields
\begin{align*}
\mathbb{E}[X_1 X_2 \cdot \mathds{1}_{\{\tau_{1,2}>a\}}] &= a_1 a_2 \mathcal{M}_{(2,0,0),\bm{0},(a,0,0)} + a_1 \mathcal{M}_{(1,0,1),\bm{0},(a,0,0)} \\
&\quad + a_2 \mathcal{M}_{(1,1,0),\bm{0},(a,0,0)} + \mathcal{M}_{(0,1,1),\bm{0},(a,0,0)}.
\end{align*}

\subsubsection{Common-Shock Multivariate Tail Covariance ($\mbox{MTCov}^{\text{CS}}$)}
The multivariate tail covariance ($\mbox{MTCov}$) provides a measure of dependence in the tail of a multivariate distribution, indicating how risks co-vary under extreme conditions \citep{landsman2018multivariate}. Analogous to the common-shock conditioning for other measures, we define the Common-Shock Multivariate Tail Covariance, $\mbox{MTCov}^{\text{CS}}_{a}(X_1, X_2)$. This quantifies the covariance between $X_1$ and $X_2$ conditioned solely on the common-shock time $\tau_{1,2}$ exceeding a threshold $a \ge 0$ as
\[
\mbox{MTCov}^{\text{CS}}_{a}(X_1, X_2) = \mbox{Cov}(X_1, X_2 \mid \tau_{1,2} > a).
\]
This simplifies to
\[
\mbox{MTCov}^{\text{CS}}_{a}(X_1, X_2) = \mathbb{E}[X_1 X_2 \mid \tau_{1,2}>a] - \mathbb{E}[X_1\mid \tau_{1,2}>a]\mathbb{E}[X_2\mid \tau_{1,2}>a].
\]
$\mbox{MTCov}^{\text{CS}}_{a}(X_1, X_2)$ is thus useful for understanding how the linear dependence between $X_1$ and $X_2$ changes based on the information that the common shock has been delayed past time $a$.
The term $\mathbb{E}[X_1 X_2 \mid \tau_{1,2}>a]$ is the $\mbox{MTCE}^{\text{CS}}_{a}(X_1, X_2)$ previously derived, while the conditional expectations $\mathbb{E}[X_i\mid \tau_{1,2}>a]$ are precisely the $\mbox{CV@R}^{\text{CS}}_a(X_i)$ terms introduced above.
These components can then be combined to compute the common-shock multivariate tail covariance.

Figure~\ref{fig:cstc} displays the behavior of $\mbox{MTCE}^{\text{CS}}_a$ (left panel) and $\mbox{MTCov}^{\text{CS}}_a$ (right panel) for the CSPH model introduced in Example~\ref{ex:risk_measure}, evaluated over a range of values for the threshold $a$. As $a$ increases, we observe that $\mbox{MTCE}^{\text{CS}}_a$ increases, reflecting a higher conditional cross expectation for larger realizations of the common shock. In contrast, $\mbox{MTCov}^{\text{CS}}_a$ decreases with $a$, indicating a reduction in conditional covariance as the common shock becomes more dominant, showing that idiosyncratic variation plays a smaller role in the dependence structure.

\begin{figure}[!htbp]
\centering
\includegraphics[width=0.49\textwidth]{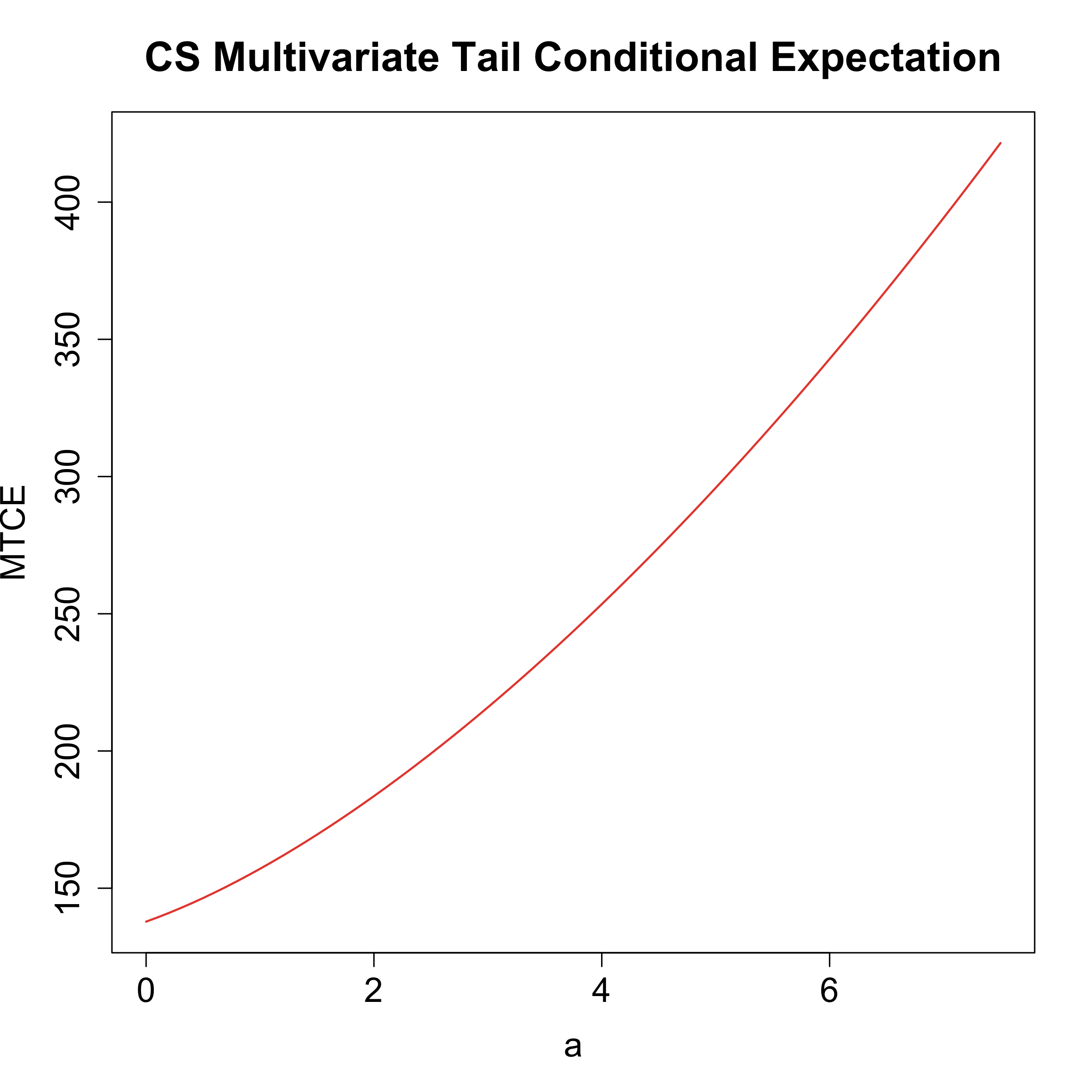}
\includegraphics[width=0.49\textwidth]{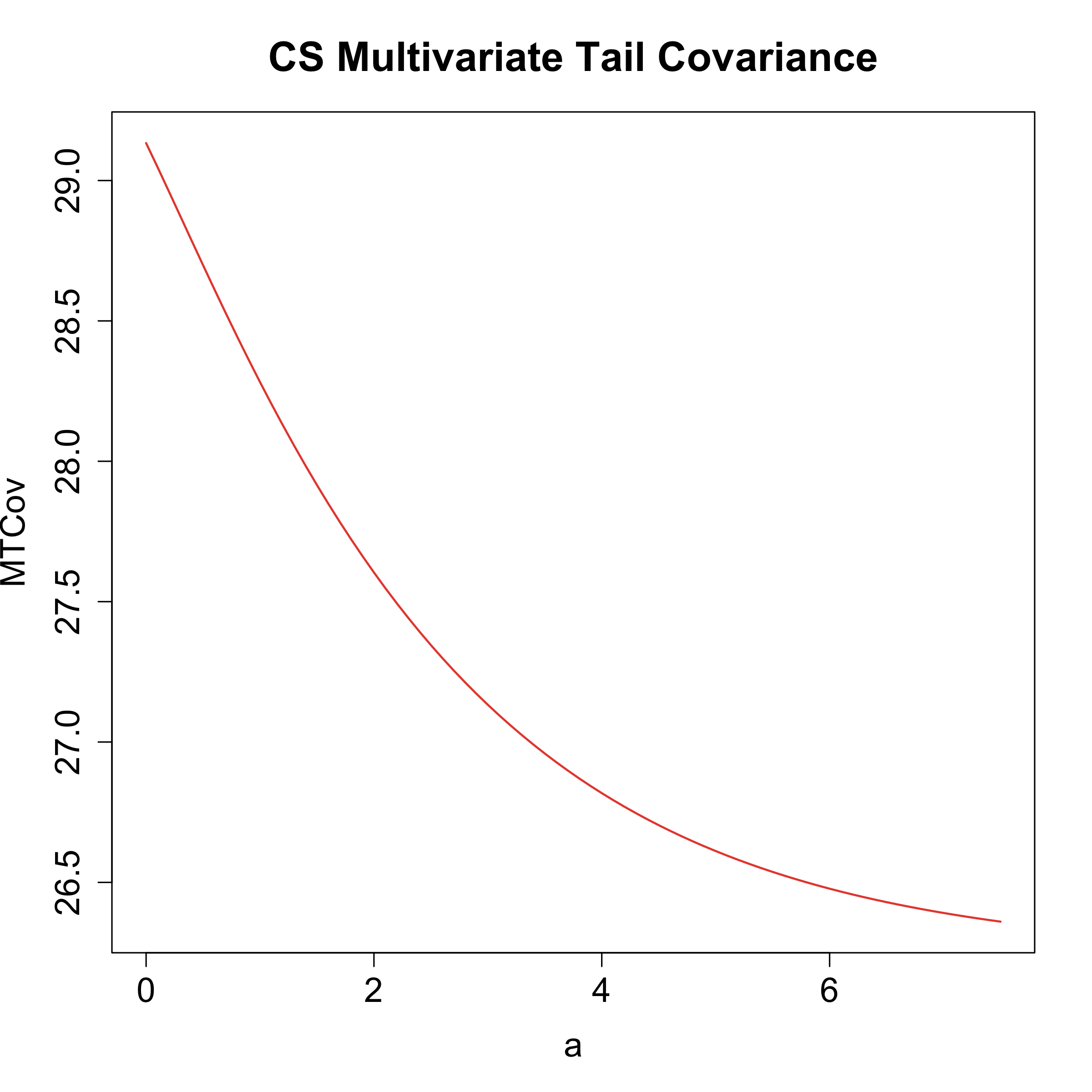}
\caption{$\mbox{MTCE}^{\text{CS}}_a$ (left) and $\mbox{MTCov}^{\text{CS}}_a$ (right) for the CSPH model in Example~\ref{ex:risk_measure} plotted for different values of the threshold $a$.}
\label{fig:cstc}
\end{figure}

\section{Conditional Dependence Measures}
\label{sec:conditional_dependence}

While the $\mbox{CSPH}$ framework models dependence through the common-shock mechanism and the shared state at the shock's occurrence, it is also insightful to analyze the dependence structure that persists between the components $X_1$ and $X_2$ after accounting for the specific realization of the common-shock time $\tau_{1,2}$. By conditioning on $\tau_{1,2}=t \ge 0$, the random variables become $X_1(t) = a_1 t + \tilde{\tau}_1$ and $X_2(t) = a_2 t + \tilde{\tau}_2$. Any remaining dependence between $X_1(t)$ and $X_2(t)$ for a fixed $t \ge 0$ is solely determined by the dependence between the residual lifetimes $\tilde{\tau}_1$ and $\tilde{\tau}_2$. This section explores this conditional dependence structure.

Given that the common shock occurs at time $\tau_{1,2}=t \ge 0$ and the system transitions into state $K=k \in \mathcal{S}$, the residual lifetimes $\tilde{\tau}_1$ and $\tilde{\tau}_2$ are conditionally independent, following $\mbox{PH}$ distributions $\mbox{PH}(\bm{e}_k^\top, \bm{Q}_1)$ and $\mbox{PH}(\bm{e}_k^\top, \bm{Q}_2)$, respectively. The probability of transitioning into state $k \in \mathcal{S}$ at time $t \ge 0$, given $\tau_{1,2}=t$, is
\[
\pi_k(t) := \P(K=k \mid \tau_{1,2}=t) = \frac{\bm{\alpha}e^{\bm{T}t}\bm{u}^{(k)}}{\bm{\alpha}e^{\bm{T}t}\bm{U}\bm{1}}.
\]

Thus, the joint probability density function of $(\tilde{\tau}_1, \tilde{\tau}_2)$ given $\tau_{1,2}=t \ge 0$ is a mixture given by
\[
f_{\tilde{\tau}_1, \tilde{\tau}_2 \mid \tau_{1,2}=t}(y_1, y_2) = \sum_{k \in \mathcal{S}} \pi_k(t) f_{\tilde{\tau}_1 \mid K=k}(y_1) f_{\tilde{\tau}_2 \mid K=k}(y_2),
\]
where $f_{\tilde{\tau}_1 \mid K=k}(y_1) = \bm{e}_k^\top e^{\bm{Q}_1 y_1} \bm{q}_1$ and $f_{\tilde{\tau}_2 \mid K=k}(y_2) = \bm{e}_k^\top e^{\bm{Q}_2 y_2} \bm{q}_2$ are the densities of $\mbox{PH}(\bm{e}_k^\top, \bm{Q}_1)$ and $\mbox{PH}(\bm{e}_k^\top, \bm{Q}_2)$, respectively. This corresponds to a bivariate density in the \(\mbox{mPH}\) class introduced in \cite{bladt2023tractable}. Moreover, since this density is of the form
\[f(x_1, x_2) = \sum_{k \in \mathcal{S}} p_k f_{k_1}(x_1)f_{k_2}(x_2),\]
with $p_k = \pi_k(t)$, and $f_{k1}(y_1) = f_{\tilde{\tau}_1 \mid K=k}(y_1)$ and $f_{k2}(y_2) = f_{\tilde{\tau}_2 \mid K=k}(y_2)$  being matrix-exponential densities (PH distributions are a subclass of the matrix-exponential family; see \citealp{bladt2017matrix}), it also belongs to the broader class of multivariate matrix-exponential affine mixtures (\(\mbox{MMEam}\)) introduced by \cite{cheung2022multivariate}. The framework of \cite{cheung2022multivariate} can therefore be directly applied.
This enables the derivation of expressions for various dependence measures for $(\tilde{\tau}_1, \tilde{\tau}_2)$ conditional on the shock time $\tau_{1,2}=t$. These measures, which depend on $t$ through $\pi_k(t)$, characterize the residual dependence structure. Note that for linear and rank correlation measures, the dependence between $(X_1(t), X_2(t))$ is the same as that between $(\tilde{\tau}_1, \tilde{\tau}_2)$ given $\tau_{1,2}=t$.

Table \ref{tab:conditional_dependence_measures} summarizes formulas for selected dependence measures, adapted from Theorems 3.6, 3.7, and 3.8 from \cite{cheung2022multivariate}. The table is not exhaustive, and most dependence measures present in \cite{cheung2022multivariate} can be adapted to study the behavior of $(\tilde{\tau}_1, \tilde{\tau}_2)$ given $\tau_{1,2}=t \ge 0$. These conditional dependence measures, which are functions of $t \ge 0$ through the mixing probabilities $\pi_k(t)$, provide a nuanced understanding of how the relationship between the residual components $\tilde{\tau}_1$ and $\tilde{\tau}_2$ evolves depending on the duration of the pre-shock phase.

\begin{table}[htbp]
\centering
\caption{Conditional Dependence Measures for $(\tilde{\tau}_1, \tilde{\tau}_2)$ given $\tau_{1,2}=t$.}
\label{tab:conditional_dependence_measures}
\begin{tabular}{ll}
\toprule
Measure & Formula \\
\midrule
$\mathbb{E}[\tilde{\tau}_i \mid \tau_{1,2}=t]$ & $\sum_{k \in \mathcal{S}} \pi_k(t) \mathbb{E}_k[\tilde{\tau}_i]$ \\
& where $\mathbb{E}_k[\tilde{\tau}_i] = \bm{e}_k^\top (-\bm{Q}_i)^{-2}\bm{q}_i$ \\[1ex]
$\text{Var}(\tilde{\tau}_i \mid \tau_{1,2}=t)$ & $\left( \sum_{k \in \mathcal{S}} \pi_k(t) \mathbb{E}_k[\tilde{\tau}_i^2] \right) - \left( \mathbb{E}[\tilde{\tau}_i \mid \tau_{1,2}=t] \right)^2$ \\
& where $\mathbb{E}_k[\tilde{\tau}_i^2] = 2\bm{e}_k^\top (-\bm{Q}_i)^{-3}\bm{q}_i$ \\[1ex]
$\mathbb{E}[\tilde{\tau}_1 \tilde{\tau}_2 \mid \tau_{1,2}=t]$ & $\sum_{k \in \mathcal{S}} \pi_k(t) \mathbb{E}_k[\tilde{\tau}_1] \mathbb{E}_k[\tilde{\tau}_2]$ \\[1ex]
Pearson's $\rho(\tilde{\tau}_1, \tilde{\tau}_2 \mid \tau_{1,2}=t)$ & $\frac{\mathbb{E}[\tilde{\tau}_1 \tilde{\tau}_2 \mid \tau_{1,2}=t] - \mathbb{E}[\tilde{\tau}_1 \mid \tau_{1,2}=t]\mathbb{E}[\tilde{\tau}_2 \mid \tau_{1,2}=t]}{\sqrt{\text{Var}(\tilde{\tau}_1 \mid \tau_{1,2}=t)\text{Var}(\tilde{\tau}_2 \mid \tau_{1,2}=t)}}$ \\[1ex]
Kendall's $\tau(\tilde{\tau}_1, \tilde{\tau}_2 \mid \tau_{1,2}=t)$ & $4 \sum_{k \in \mathcal{S}} \sum_{m \in \mathcal{S}} \pi_k(t) \pi_m(t) c_{km}^{(1)} c_{km}^{(2)} - 1$ \\
& where $c_{km}^{(i)} = \int_0^\infty F_{\tilde{\tau}_i \mid K=k}(y) f_{\tilde{\tau}_i \mid K=m}(y) \dd y$ \\
& $c_{km}^{(i)} = 1 - (\bm{e}_k^\top \otimes \bm{e}_m^\top) (-(\bm{Q}_i \oplus \bm{Q}_i))^{-1} ( \bm{1} \otimes \bm{q}_i )$  \\[1ex]
Spearman's $\rho_S(\tilde{\tau}_1, \tilde{\tau}_2 \mid \tau_{1,2}=t)$ & $12 \sum_{k \in \mathcal{S}} \sum_{m \in \mathcal{S}} \sum_{n \in \mathcal{S}} \pi_k(t) \pi_m(t) \pi_n(t) c_{km}^{(1)} c_{kn}^{(2)} - 3$  \\
\bottomrule
\end{tabular}
\end{table}

Figure~\ref{fig:conmean} shows the conditional expectations $\mathbb{E}[\tilde{\tau}_i \mid \tau_{1,2}=t]$ for $i = 1, 2$, under the CSPH model specified in Example~\ref{ex:risk_measure}, plotted as a function of $t$. For the first margin (left), we observe a decreasing trend, indicating that the conditional mean of the idiosyncratic component $\tilde{\tau}_1$ becomes smaller as the duration of the common shock increases. In contrast, the curve for the second margin (right) remains flat. This behavior arises from the structure of the $\mathbf{Q}_2$ matrix, as the evolution of $\tilde{\tau}_2$ is independent of the state in which the common-shock component terminates.

\begin{figure}[!htbp]
\centering
\includegraphics[width=0.49\textwidth]{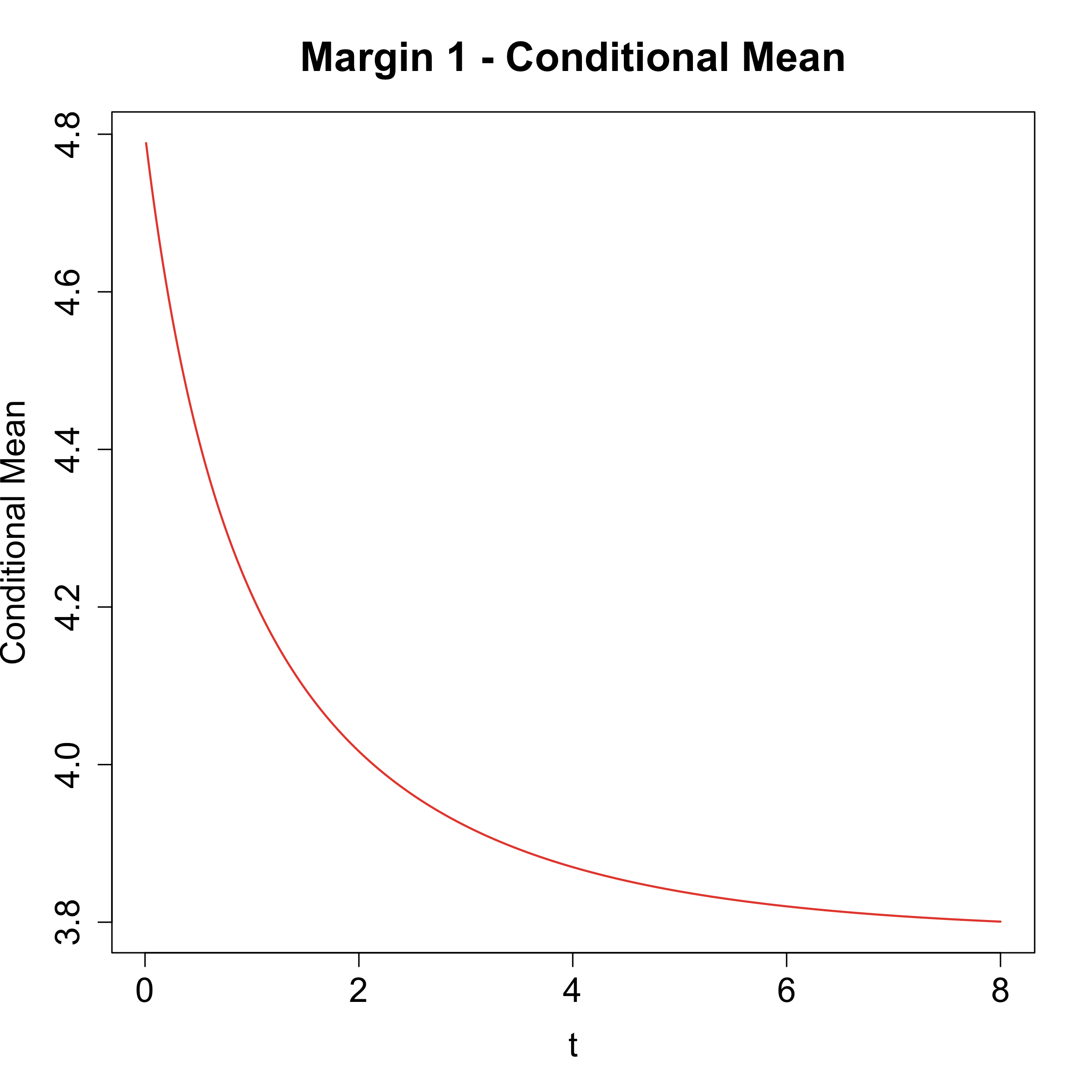}
\includegraphics[width=0.49\textwidth]{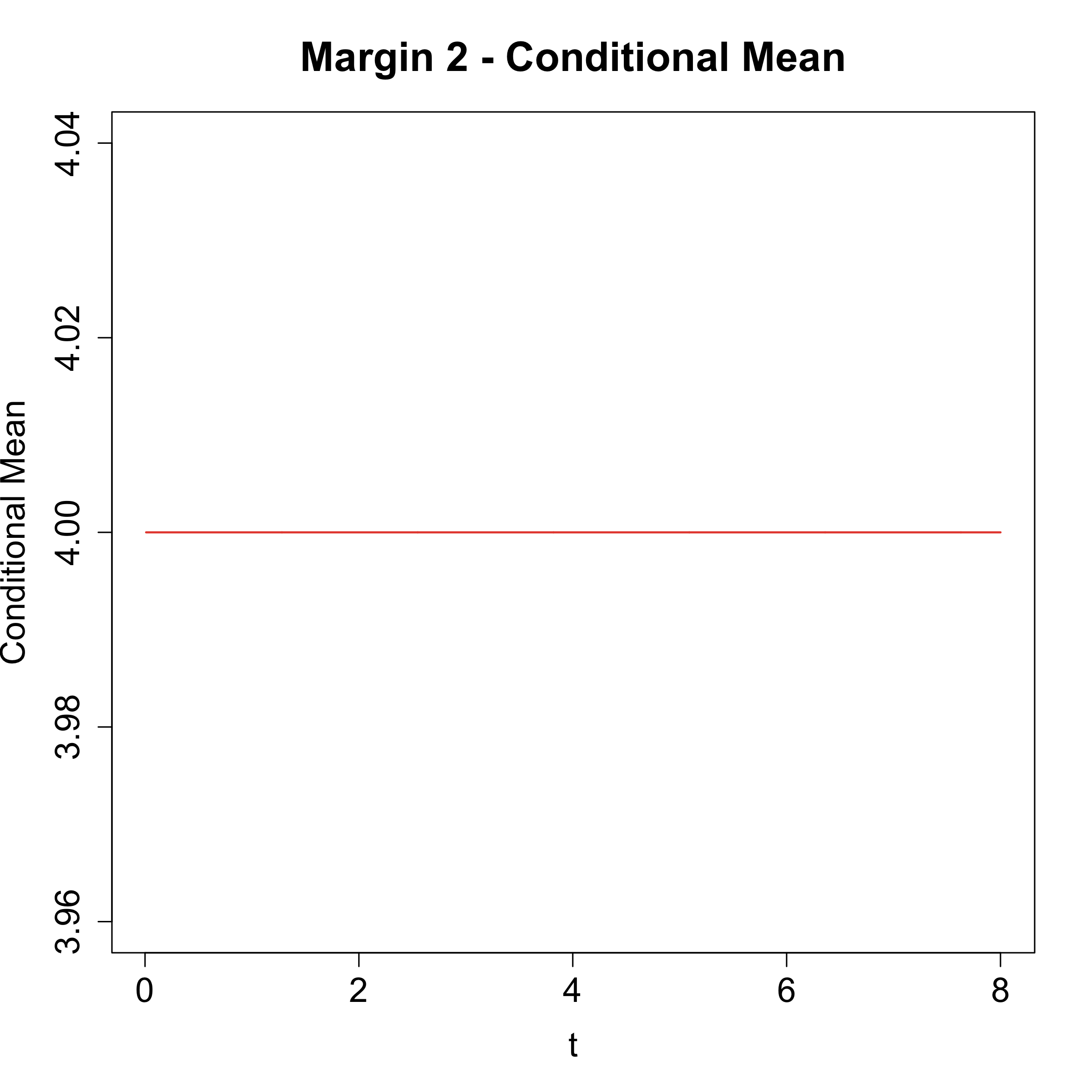}
\caption{Conditional means, $\mathbb{E}[\tilde{\tau}_i \mid \tau_{1,2}=t]$, $i =1, 2$, for the CSPH model in Example~\ref{ex:risk_measure} plotted as a function of $t$.}
\label{fig:conmean}
\end{figure}

\section{Statistical Inference and Numerical Examples}\label{sec:inference_examples} 

Parameter estimation for the $\mbox{CSPH}$ model is carried out via maximum likelihood estimation. While the Expectation-Maximization (EM) algorithm is a standard tool for latent variable models, such as $\mbox{PH}$ distributions \citep[cf.][]{asmussen1996fitting}, its application to the $\mbox{CSPH}$ framework, although feasible, is cumbersome and slow. The formulation of the ``complete data" and the subsequent E-step, which involves computing the expectation of the complete-data loglikelihood given the observed data, requires computing conditional expectations of the unobserved path statistics related to the coupled pre-shock phase and the separate post-shock phases. This added complexity makes the implementation of the algorithm substantially more involved than in standard $\mbox{PH}$ models.

We therefore employ direct numerical maximization of the loglikelihood function using gradient-based methods. Specifically, we use the L-BFGS optimization routine as implemented in R throughout the \texttt{optim()} function. To enhance performance on large datasets, we parallelize the likelihood evaluation using the \texttt{parallel} R package.
Finally, to reduce the number of parameters and improve numerical stability, we adopt an equivalent representation of the model in which only the first margin is scaled:
\begin{align*}
X_1 = \beta ( \tau_{1,2} + \tilde{\tau}_1 ), \quad
X_2 = \tau_{1,2} + \tilde{\tau}_2,
\end{align*}
with $\beta >0$.
This is equivalent to the original $\mbox{CSPH}$ specification since the class of $\mbox{PH}$ distributions is closed under positive scalar multiplication, which corresponds to an inverse scaling of the associated subintensity matrices. This reparameterization reduces the parameter space without sacrificing model flexibility.
{
The R code used for all numerical examples presented in this paper is publicly available in the Zenodo repository \url{https://doi.org/10.5281/zenodo.17624706}.
}



\subsection{Synthetic Data}
We consider a simulated example where the data comes from a $\mbox{CSPH}$ model. The dataset comprises 2,000 iid bivariate observations drawn from the model specified in Example~\ref{ex:risk_measure}.
To illustrate the effectiveness of the proposed estimation method in recovering this structure, we fit a $\mbox{CSPH}$ model with the same dimensions, namely $|\mathcal{E}| = 3$ and $|\mathcal{S}| = 2$. The resulting estimated parameters are as follows:
\begin{gather*}
	\hat{\bfalp}=\left(
	1, \,0,\, 0 \right)\,, \quad
	\hat{\bfT}=\left( \begin{array}{cccc}
	-1.0223  & 0.4276 & 0.5147 \\
	13.7666 & -16.4783  & 1.2234  \\
	 0.0565  & 0.0765 & -0.4225
	\end{array} \right) \,, 	\quad
	\hat{\bfU}=\left( \begin{array}{cc}
	0.0005 & 0.0795  \\
	0.1797 &  1.3086   \\
	0.2548 & 0.0347
	\end{array} \right) \,, \\
	\hat{\bfQ}_1=\left( \begin{array}{cc}
	-0.5613 & 0.0006  \\
	0.7939 & -0.7948
	\end{array} \right) \,, \quad
	\hat{\bfQ}_2=\left( \begin{array}{cc}
	-11.8073 & 11.5334  \\
	 9.9585 & -10.1935
	\end{array} \right) \,, \quad
	\hat{\beta}= 1.9344 \,.
\end{gather*}

The fit quality is illustrated in Figure~\ref{fig:sim_joint}, which combines a scatter plot of the simulated data, contour curves of the fitted $\mbox{CSPH}$ model, and histograms of the marginal distributions. The contour lines closely follow the structure of the data cloud, indicating that the model captures the joint distribution of the data. Additionally, the histograms in the top and right panels show that the fitted model accurately reproduces the marginal behavior of each component. In both cases, the fitted marginal densities closely align with the original specification. Moreover, Figure~\ref{fig:sim_com} includes a histogram of the simulated common shock, where we observe that the fitted model successfully recovers this structure and aligns closely with the original specification.

The agreement between the empirical and fitted moments further supports the model's adequacy. For the estimated $\mbox{CSPH}$ model, $\mathbb{E}[X_1] = 13.19$, $\mathbb{E}[X_2] = 8.57$, $\mathbb{E}[\tau_{1,2}] = 4.62$, and a correlation of $\rho(X_1, X_2) = 0.6423$. These are close to the empirical values of $13.19$, $8.57$, $4.61$, and $0.6477$, respectively. The loglikelihood of the fitted model is $-12,183.50$, compared to $-12,185.40$ for the true (data-generating) parameters, which also indicates a high-quality fit.

\begin{figure}[!htbp]
\centering
\includegraphics[width=\textwidth]{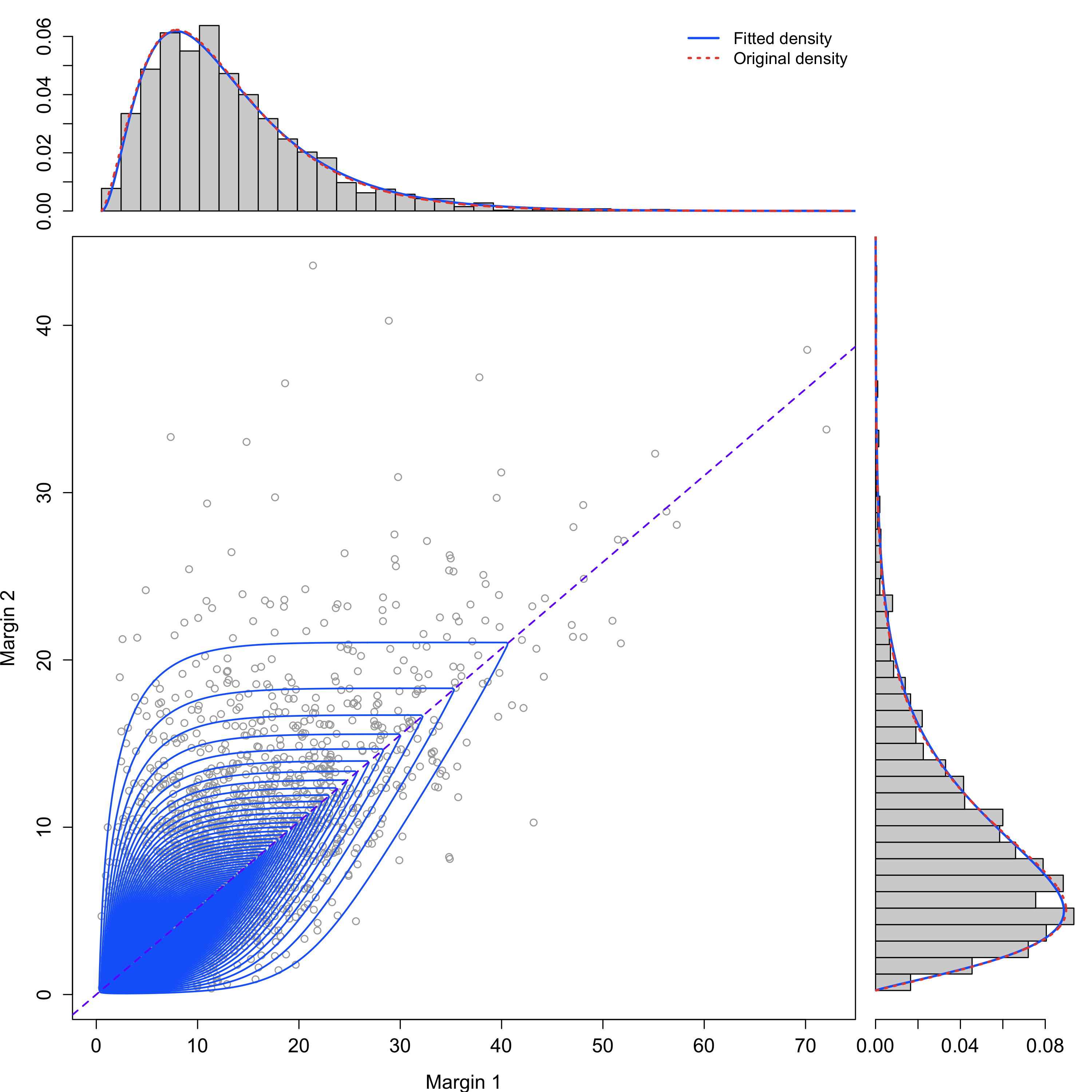}
\caption{Fit quality assessment for the synthetic data example. Center: Scatter plot of 2,000 simulated bivariate observations from a CSPH model with $|\mathcal{E}| = 3$ transient states and $|\mathcal{S}| = 2$ post-shock states, overlaid with contour curves of the fitted CSPH joint density. Top and right: Histograms of the simulated margins along with the corresponding marginal densities of the fitted CSPH model. The close alignment between fitted and true distributions indicates successful parameter recovery.}
\label{fig:sim_joint}
\end{figure}

\begin{figure}[!htbp]
\centering
\includegraphics[width=0.5\textwidth]{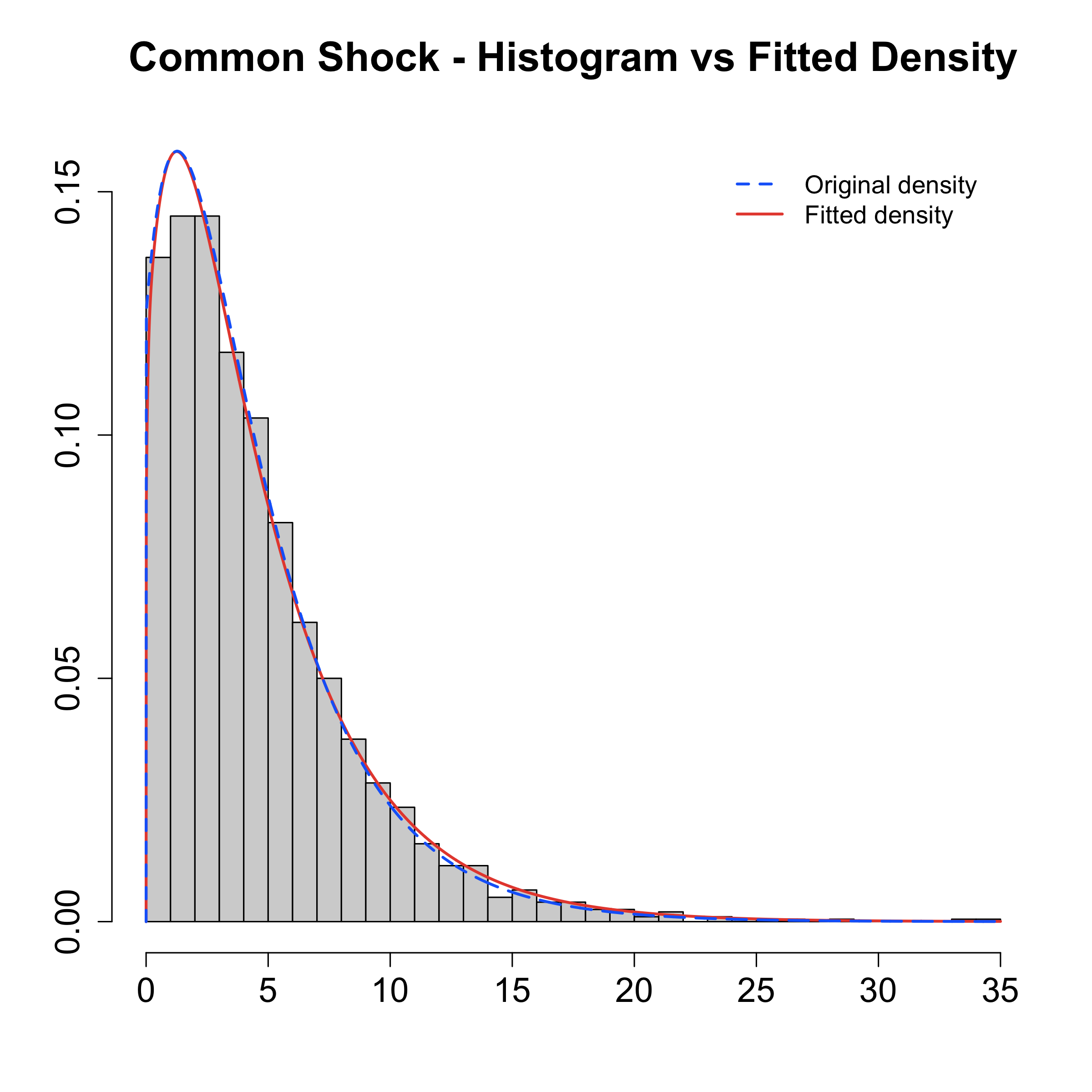}
\caption{Common-shock component recovery assessment for the synthetic data example. Histogram of the simulated common-shock component $\tau_{1,2}$ along with its corresponding fitted density from the estimated CSPH model. The fitted density successfully captures the distribution of the latent common-shock, demonstrating the model's ability to identify this unobserved component from bivariate data.}
\label{fig:sim_com}
\end{figure}

\subsection{Danish Fire Insurance Data}

We now consider the well-known Danish fire insurance dataset. We focus on the logarithm of the claim sizes for the building and content components, restricted to the domain $(1, \infty) \times (1, \infty)$. A $\mbox{CSPH}$ model with $|\mathcal{E}| = 3$ and $|\mathcal{S}| = 2$ is fitted to the transformed data, obtaining the following estimated parameters:
\begin{gather*}
	\hat{\bfalp}=\left(
	0.0006, \, 0.3728, \, 0.6266 \right)\,, \\
	\hat{\bfT}=\left( \begin{array}{cccc}
	-1.9164 &  0.0006   &   0.0069 \\
	 1.8615 & -1.8626   &   0.0010  \\
	 10.4880 & 168.3337 & -16088.4190
	\end{array} \right) \,, 	\quad
	\hat{\bfU}=\left( \begin{array}{cc}
	0.0009   &  1.9081  \\
	0.0002   & 0.0000  \\
	1532.0365 & 14377.5609
	\end{array} \right) \,, \\
	\hat{\bfQ}_1=\left( \begin{array}{cc}
	-1.1644 & 0.0002  \\
	0.8706 & -1.1738
	\end{array} \right) \,, \quad
	\hat{\bfQ}_2=\left( \begin{array}{cc}
	-2.0825 & 0.0004  \\
	 1.3176 & -2.1302
	\end{array} \right) \,, \\
	\hat{\beta}= 0.5763 \,.
\end{gather*}

The overall quality of the fit is illustrated in Figure~\ref{fig:danish_joint}. We observe that the contour curves of the fitted $\mbox{CSPH}$ model closely align with the scatter of the observed data, indicating that the joint distribution is well captured. The histograms of the log-transformed building and content claim sizes are well approximated by the corresponding fitted marginal densities, suggesting an adequate modeling of the margins.
\begin{figure}[!htbp]
\centering
\includegraphics[width=\textwidth]{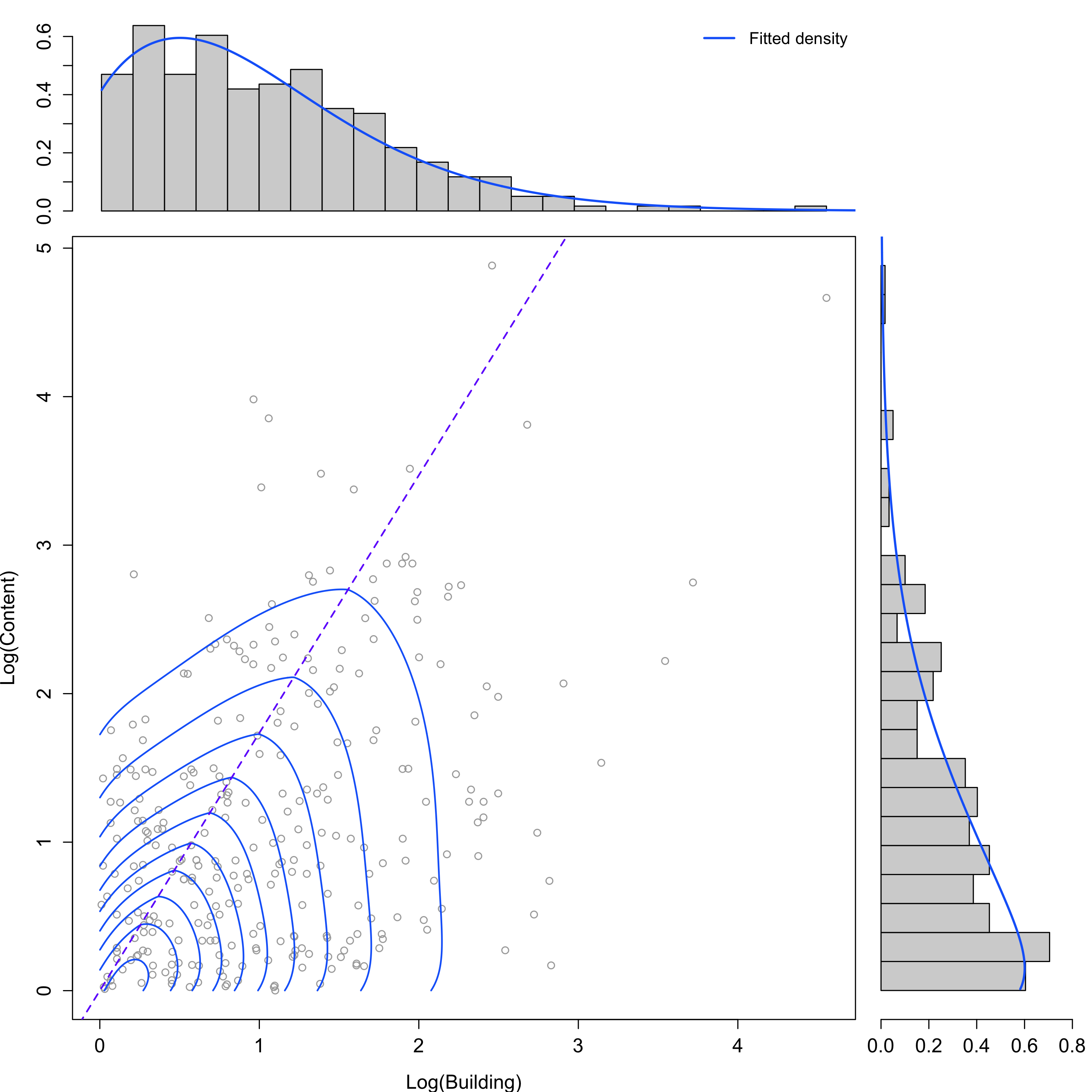}
\caption{ Fit quality assessment for the Danish fire insurance dataset. Center: Scatter plot of log-transformed building and content claim sizes overlaid with contour curves of the fitted CSPH joint density (CSPH model with $|\mathcal{E}| = 3$ and $|\mathcal{S}| = 2$). Top and right: Histograms of the logarithmic claim sizes for building (top) and content (right) along with the corresponding marginal densities of the fitted model. The contour curves closely follow the data cloud structure, and the marginal fits accurately capture the empirical distributions, indicating good overall fit quality.}
\label{fig:danish_joint}
\end{figure}

Figure~\ref{fig:danish_com} provides insights into the common-shock component of the fitted distribution. The fitted distribution function (right panel) suggests a distribution with a non-negligible atom at zero. This behavior reflects a structural feature in the data whereby not all events are accompanied by a common shock. Nevertheless, the common shock cannot be considered negligible. For instance, its estimated mean is approximately $0.40$, compared to $1.07$ and $1.15$ for the building and content components, respectively. This highlights the substantial role of the common shock in the overall risk structure and its significant contribution to the dependence observed between the two margins.
\begin{figure}[!htbp]
\centering
\includegraphics[width=0.49\textwidth]{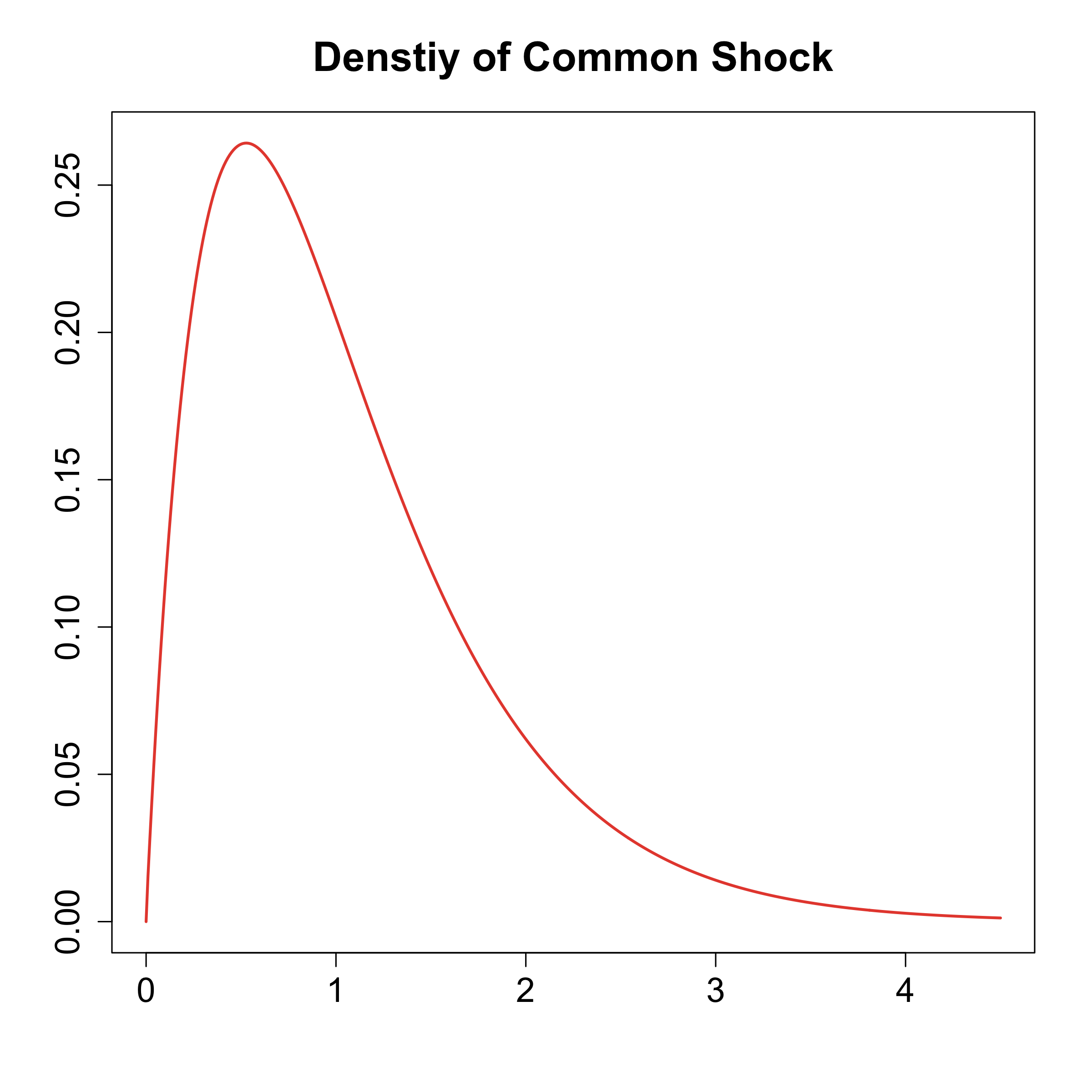}
\includegraphics[width=0.49\textwidth]{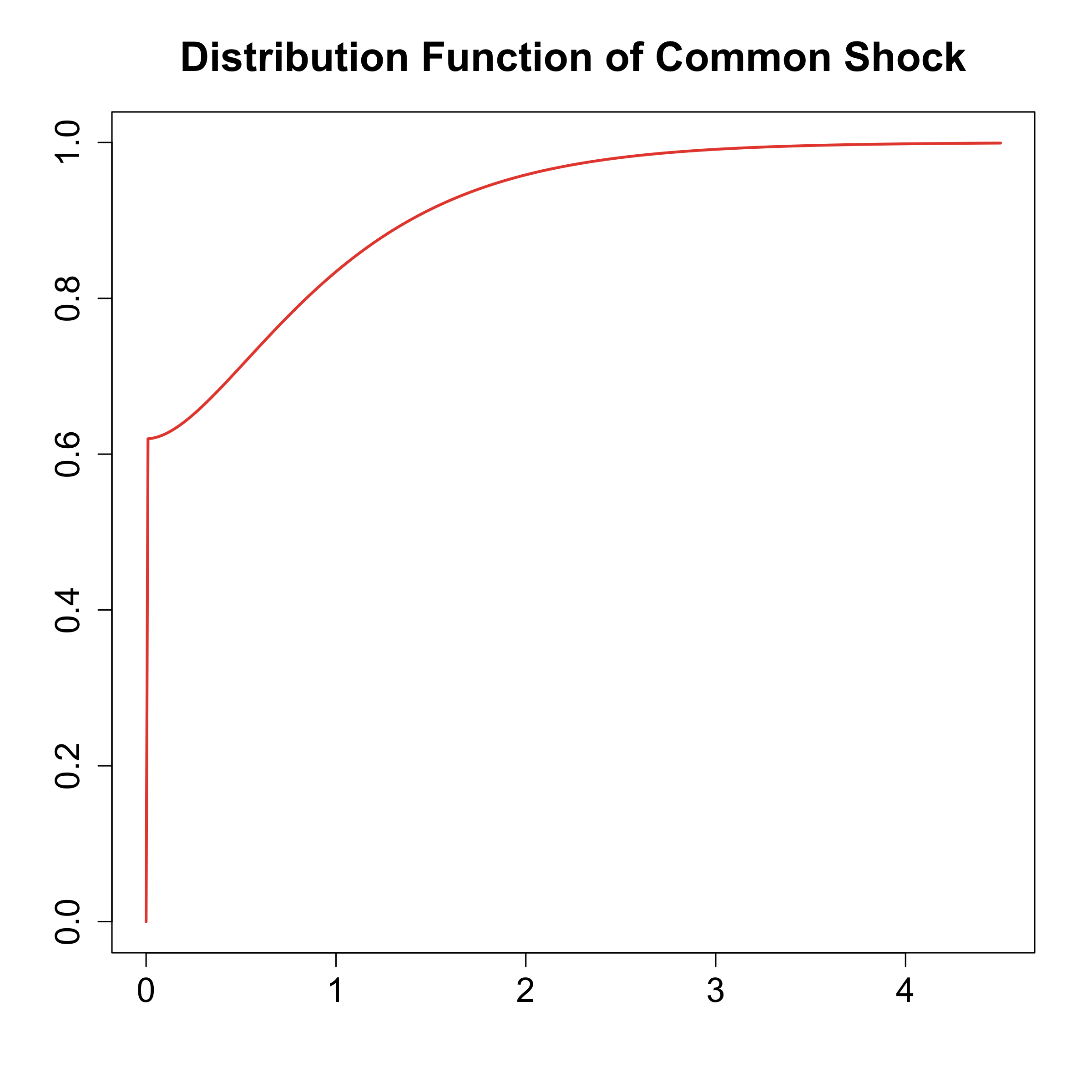}
\caption{ Common-shock component for the Danish fire insurance fitted model. Left: Density function of the estimated common-shock component $\tau_{1,2}$, showing a distribution with substantial probability mass near zero. Right: Distribution function revealing a non-negligible atom at zero, indicating that not all claims involve a common shock. Despite this, the common shock remains an important component with estimated mean 0.40, compared to marginal means of 1.07 (building) and 1.15 (content), highlighting its role in driving dependence.}
\label{fig:danish_com}
\end{figure}

With the fitted model at hand, we proceed to the computation of several risk measures. We begin by examining the $\mathrm{CV@R}$ in Figure~\ref{fig:cvar_danish}, where we observe the effect of the atom at zero in the distribution of the common shock, manifested as sharp jumps immediately after $a = 0$. A similar pattern is also present in the $\mbox{MTCE}^{\text{CS}}$ and $\mbox{MTCov}^{\text{CS}}$ measures shown in Figure~\ref{fig:dan_measure}, highlighting the structural role of the common shock in driving dependence within the model.

\begin{figure}[!htbp]
\centering
\includegraphics[width=0.49\textwidth]{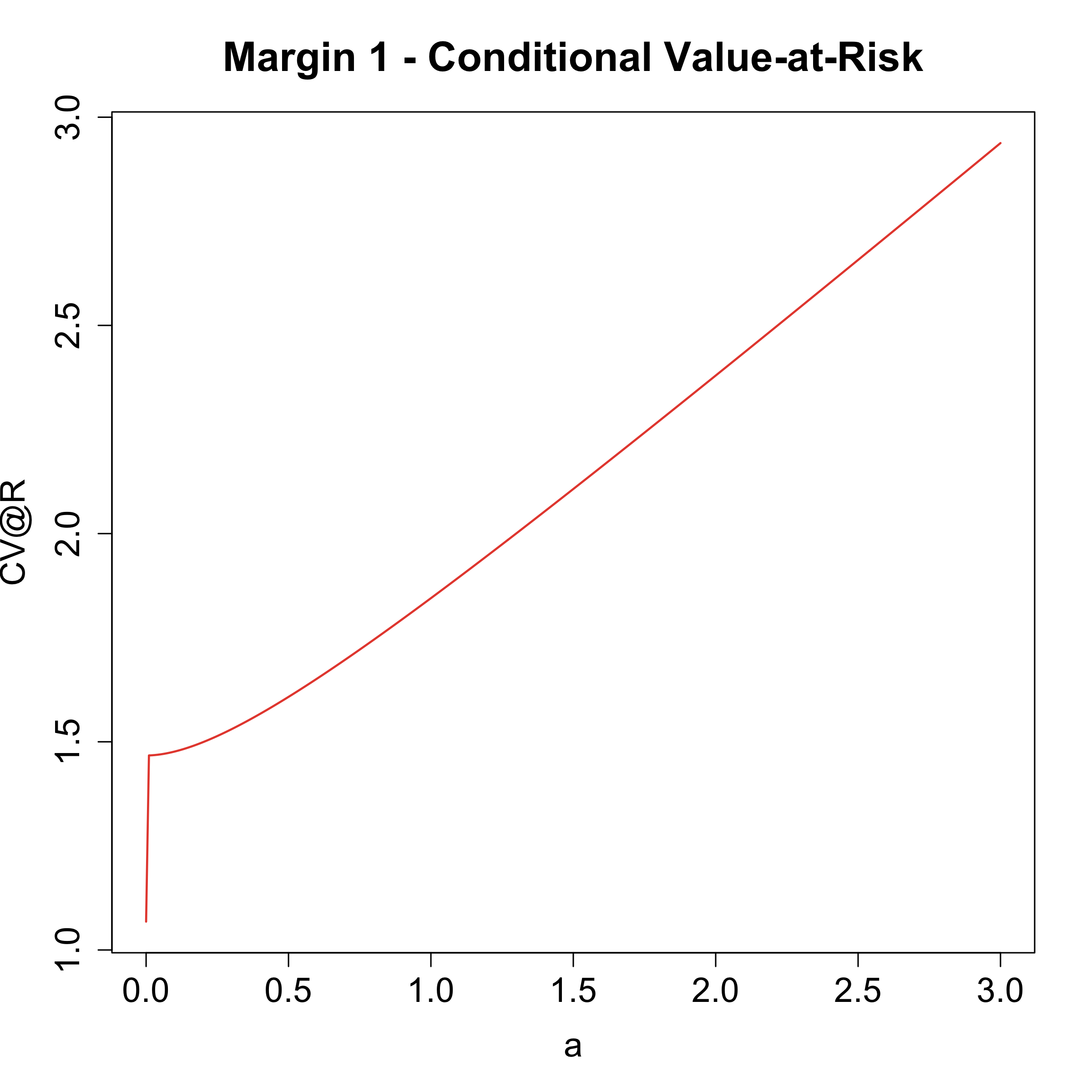}
\includegraphics[width=0.49\textwidth]{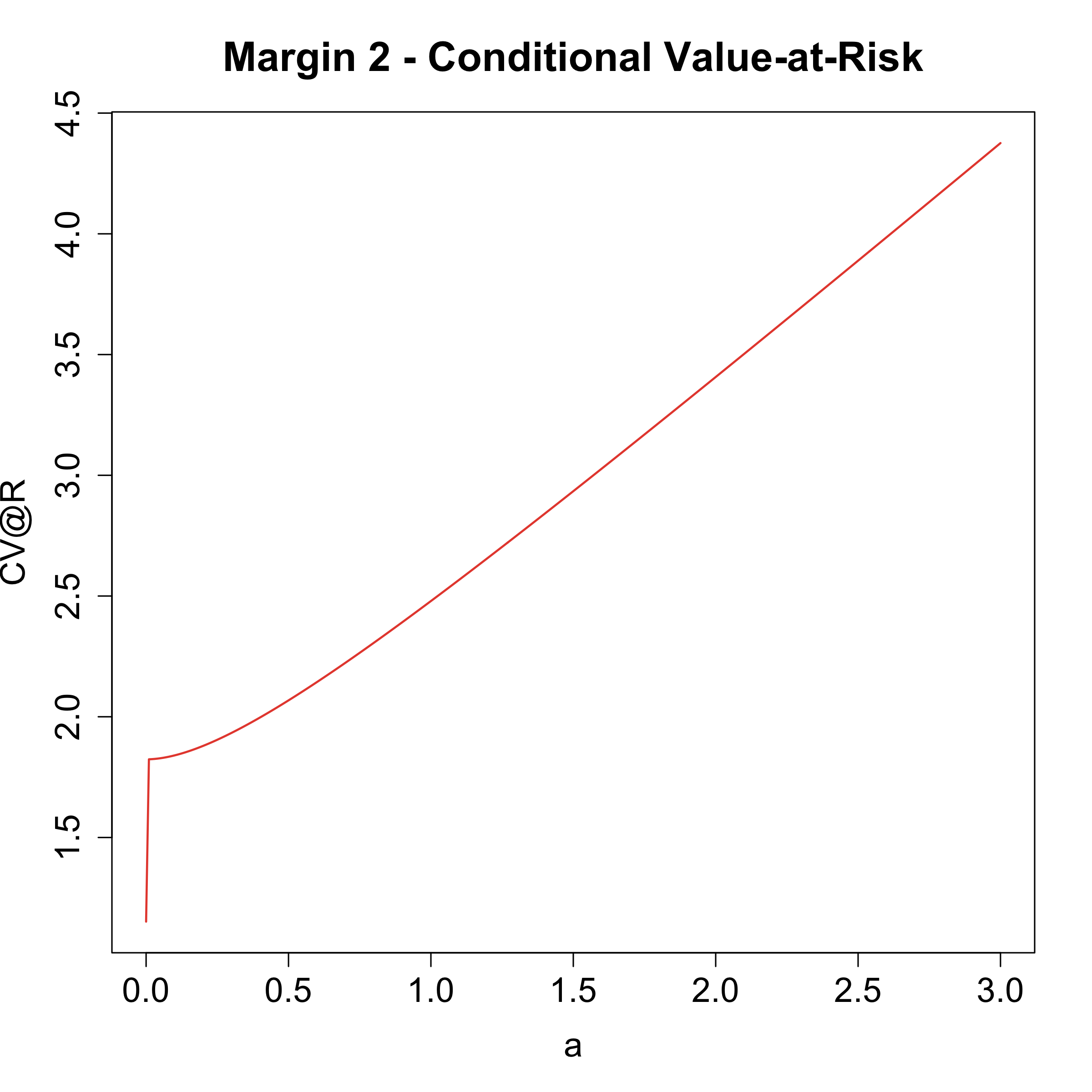}
\caption{ Conditional Value-at-Risk for the Danish fire insurance CSPH model. CV@R for the marginal distributions of building (left) and content (right) as a function of the threshold $a$ affecting the common-shock component $\tau_{1,2}$. The sharp jumps immediately after $a = 0$ reflect the atom at zero in the distribution of the common shock, demonstrating how the latent common-shock structure influences conditional risk measures.}
\label{fig:cvar_danish}
\end{figure}

\begin{figure}[!htbp]
\centering
\includegraphics[width=0.49\textwidth]{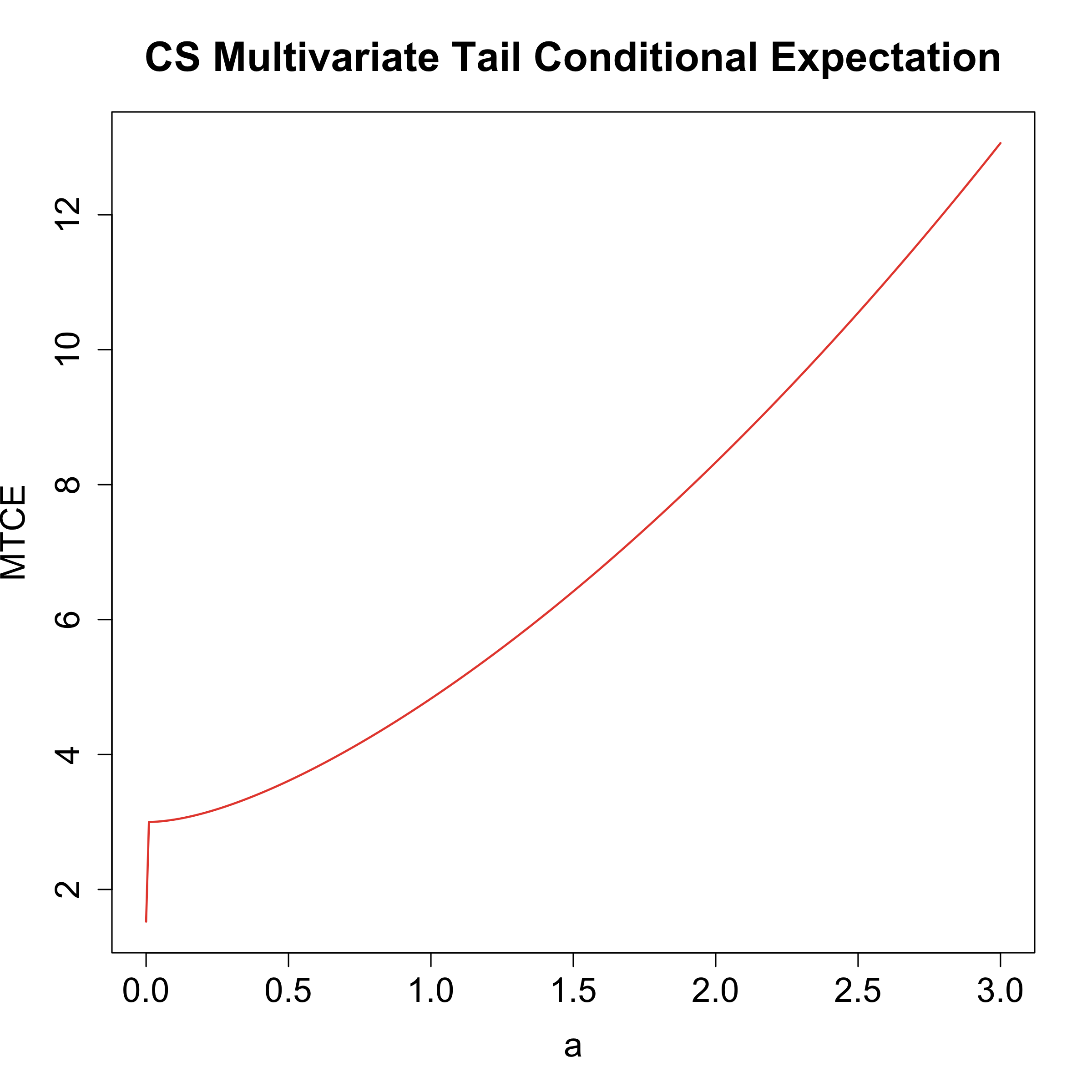}
\includegraphics[width=0.49\textwidth]{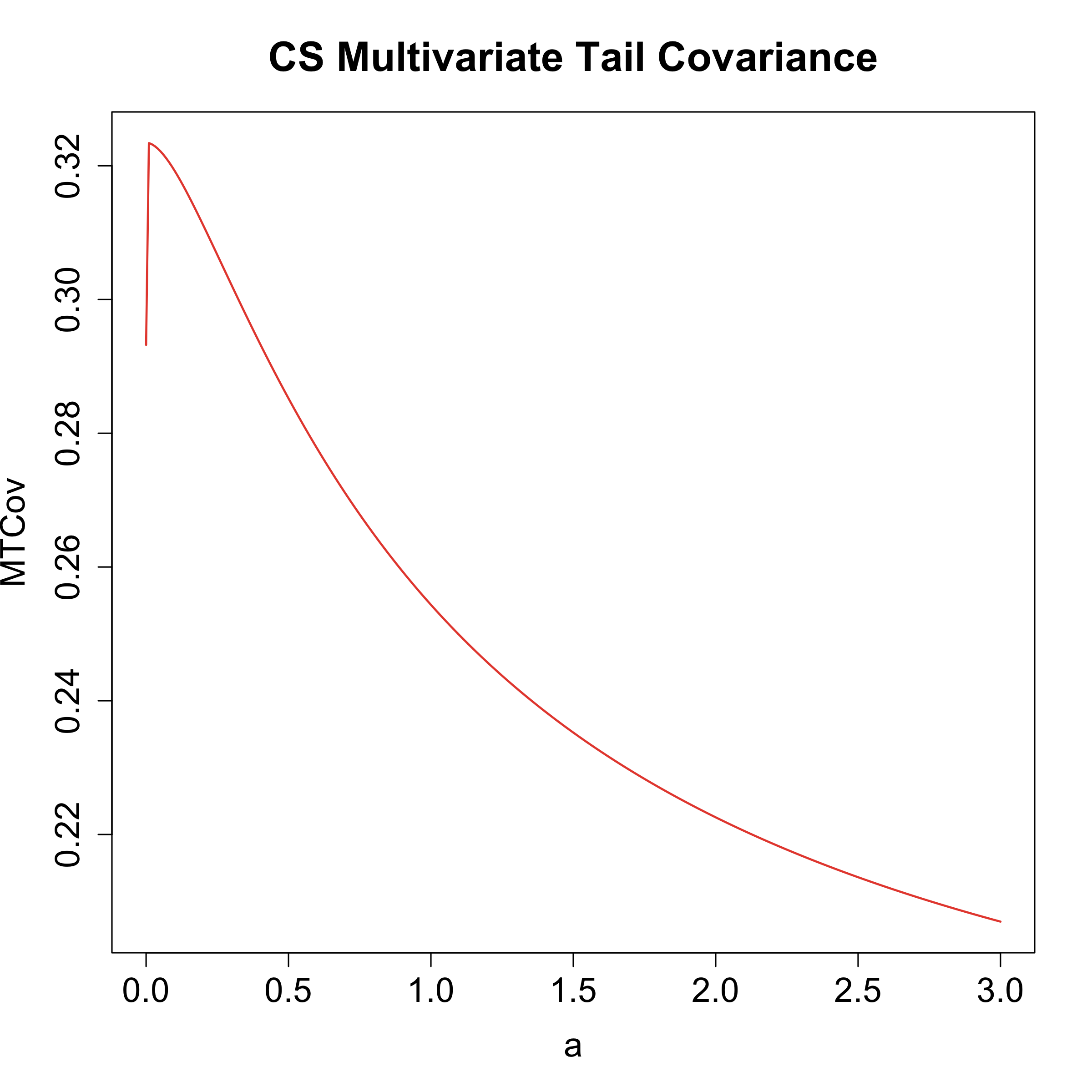}
\caption{ Conditional dependence measures for the Danish fire insurance CSPH model. Left: Mean tail conditional expectation given common shock (MTCE$^{\text{CS}}_a$) as a function of threshold $a$. Right: Mean tail conditional covariance given common shock (MTCov$^{\text{CS}}_a$) as a function of threshold $a$. Both measures exhibit sharp discontinuities after $a = 0$ due to the atom at zero in the common-shock distribution, highlighting the structural role of the common shock in driving dependence. These measures quantify how extreme common-shock events amplify the dependence between building and content claims.}
\label{fig:dan_measure}
\end{figure}

It is essential to note that fitting a $\mbox{CSPH}$ model to the log-transformed data can also be interpreted as a multivariate model for the original data with inhomogeneous $\mbox{PH}$ (IPH) margins, as introduced in \cite{albrecher2019inhomogeneous}. This connection means the fitted model induces Matrix-Pareto margins to describe the original (untransformed) data. In this setting, the tail heaviness of each marginal distribution is governed by the index of regular variation, which is determined by the real parts of the eigenvalues with the largest real part of the subintensity matrices of the marginal distributions. We obtain values of 2 and 1.85 for the building and content components, respectively. These values are in close agreement with those reported in \cite{albrecher2017reinsurance}, where a splicing approach combining a bivariate mixed Erlang distribution for the body and a bivariate Generalized Pareto distribution (GPD) for the tail yielded estimates of 2 and 1.67, respectively. A key distinction is that the splicing model in \cite{albrecher2017reinsurance} requires specifying a threshold to separate body and tail behavior, whereas the $\mbox{CSPH}$ approach achieves this transition in a fully parametric manner. We also note that a related model with matrix-Pareto margins was fitted to the same dataset in \cite{albrecher2022fitting}.

Finally, several of the risk measures discussed in this manuscript can also be computed for the multivariate matrix-Pareto model. For instance, Table~\ref{tab:var_danish} presents the $\mathrm{V@R}$ at various significance levels. Other measures, such as the $\mathrm{CV@R}$, can also be derived using the master formula~\eqref{eq:master-formula}. However, this requires a reinterpretation of the model structure, specifically, treating the common shock as acting multiplicatively on the marginal components, rather than additively.

\begin{table}[htbp]
\centering
\caption{Value-at-Risk at different levels for the CSPH model fitted to the Danish fire insurance data.}
\label{tab:var_danish}
\begin{tabular}{lccc}
\toprule
& \multicolumn{3}{c}{Level $\alpha$}   \\ \cline{2-4}
  & 0.95     & 0.975     & 0.99    \\
\midrule
Building &  13.40 & 20.64 & 35.73 \\
Content &  20.36 & 34.46 & 66.73\\
\bottomrule
\end{tabular}
\end{table}

\subsection*{Funding}
MB would like to acknowledge financial support from the Carlsberg Foundation, grant CF23-1096.

\subsection*{AI Statement}
The authors have used tools from OpenAI as a spell checker.

\subsection*{Conflict of Interest}
The authors declare no conflict of interest.

\newpage
\bibliographystyle{apalike}

\bibliography{common_shocks.bib}

\end{document}